\documentclass[journal,twoside,web]{ieeecolor}
\usepackage{}
\usepackage{generic}
\usepackage{cite}
\usepackage{amsmath,amssymb,amsfonts}
\usepackage{algorithmic}
\usepackage{graphicx}
\usepackage{textcomp}

\newtheorem{theorem}{Theorem}[section]
\newtheorem{lemma}{Lemma}[section]
\newtheorem{remark}{Remark}[section]

\newtheorem{definition}{Definition}[section]

\def\BibTeX{{\rm B\kern-.05em{\sc i\kern-.025em b}\kern-.08em
    T\kern-.1667em\lower.7ex\hbox{E}\kern-.125emX}}
\begin{document}
\title{Non-fragile Finite-time Stabilization for Discrete Mean-field Stochastic Systems}
\author{Tianliang~Zhang, Feiqi~Deng, \IEEEmembership{Senior Member, IEEE}, Peng~Shi, \IEEEmembership{Fellow, IEEE}
\thanks{Corresponding author: F. Deng.}
\thanks{T. Zhang is with School of Automation, Nanjing University of Science and Technology, Nanjing  210094, Jiangsu Province, China, and also with the School of Automation Science and Engineering, South China University of Technology, Guangzhou 510640, Guangdong Province, China.
Email:  t\_lzhang@163.com.}
\thanks{F. Deng is with the School of Automation Science and Engineering, South China University of Technology, Guangzhou 510640, Guangdong Province, China.
Email:  aufqdeng@scut.edu.cn.}
\thanks{P. Shi is with the School of Electrical and Electronic Engineering, University of Adelaide, and also with the College of Engineering and Science, Victoria University, Australia.
Email:  peng.shi@adelaide.edu.au.}}

\maketitle

\begin{abstract}
In this paper, the problem of non-fragile finite-time stabilization for linear discrete mean-field stochastic systems is studied. The uncertain characteristics in control parameters are assumed to be random satisfying the Bernoulli distribution. A new approach called
 the ``state transition
matrix method" is introduced and some necessary and sufficient conditions are derived to solve the underlying stabilization problem. The Lyapunov theorem based on the
state transition matrix also makes a contribution to the discrete finite-time control theory. One practical example is provided to validate the effectiveness of the newly proposed control strategy.
\end{abstract}


\section{Introduction}
\label{sec:introduction} \IEEEPARstart{I}{t} is well-known that the
behavior of a single individual may affect the collective action.
Conversely, in many physical or sociological dynamical processes,
collective interactions can also  change individual judgment and
behavior. In order to investigate the influence from collective to
single individual, mean-field theory has been naturally developed
\cite{application1,Pickl,application4}.
 In particular, the
high popularity of quantum computers highlights the rising
importance of mean-field theory and its relevant applications,
because the mean-field method is a common and effective method to
deal with quantum many-body problems \cite{Pickl}. The well-known
mean-field type stochastic models depict the system equation
incorporating the mean of the state variables. In recent years, many
outstanding results on the control problem relating to mean-field
type stochastic systems have been proposed in the following
literature. For example, linear-quadratic optimal control problems
were discussed in \cite{new1,Niyuanhua1,Niyuanhua2}. Lin et al.
\cite{Linyaning1}  was   concerned with Stackelberg
game issue for mean-field stochastic systems. Stochastic maximum
principle was discussed in \cite{new2}. In addition, mean-field
stochastic systems with  network structure and time-delay have
attracted lots of scholars' attention, we refer the interested
readers to \cite{HuangJianhui,application7} for further references.
With respect to the stability and  stabilization problems,  Ma et
al. \cite{Malimin2} studied the mean square stability and spectral
assignment in a prescribed area region for linear discrete
mean-field stochastic (LDMFS) systems via the spectrum of a
generalized Lyapunov operator.

Finite-time stability focuses on the system state behavior only in a
specified finite-time horizon instead of the whole time interval,
which differentiates the finite-time stability from the classical
Lyapunov stability studied in \cite{JXS-RNC,ziji-auto,IJC2016} for
discrete stochastic stability  and
\cite{Khasminskii2012,maoxuerong} for stochastic stability of
continuous It\^o systems. In some practical applications, the
considered operating duration of the controlled system is often
limited \cite{JXS-IS,Qianchunjiang}, so, in some cases,  the
transient characteristics of systems may be more important than the
state convergence in an infinite-time horizon.  As it
is well-known that finite-time stability contains two kinds of
different concepts: one is defined as in
\cite{Amato2,Amato3,Amato,Lixiaodi,Niuyugang1,Yanzhiguo2,ziji-finite},
which is in fact finite-time bounded in some sense, while the other
one is defined as in
\cite{Chenweisheng,Hongyiguang,Qianchunjiang,Zhuquanxin,jinpengyu,Yinjuliang1,Yinjuliang2,Zhawenting},
where finite-time stability satisfies  both ``stability in Lyapunov
sense" and ``finite-time attractiveness".  Throughout this paper, we
study the first kind of finite-time stability and stabilization of
LDMFS systems. So from now on, when we refer to finite-time
stability  and stabilization, they are in finite-time boundedness
sense.  Finite-time stability and stabilization  have  been
researched for deterministic systems
\cite{Amato2,Amato3,Amato,Lixiaodi,Niuyugang1} and stochastic
systems \cite{Yanzhiguo2,ziji-finite}. In
\cite{Amato2,Amato3,Amato}, based on the  state transition matrix
(STM) of deterministic linear systems, necessary and sufficient
conditions have been obtained for finite-time stability and
stabilization. In \cite{Lixiaodi}, Lyapunov-type
 conditions for finite-time stability of continuous-time nonlinear  time-varying
delayed systems were presented. For continuous-time nonlinear
differential systems, \cite{Niuyugang1} proposed a suitable sliding
mode control law to drive the state trajectory into the prescribed
sliding surface within a finite time. \cite{Yanzhiguo2} and
\cite{ziji-finite} discussed the finite-time stability and
stabilization of continuous-and discrete-time stochastic systems,
respectively.  As can be seen,  most existing results on finite-time
stability and stabilization are about deterministic/stochastic
differential systems. However, regarding
 finite-time stability or stabilization for LDMFS systems, no
result has been reported so far. In fact,  we can only find few papers such as  \cite{Malimin2} to investigate
asymptotical mean square stability and stabilizability.  In
addition, most results  in stochastic systems are based on Lyapunov
function/functional method to present sufficient conditions but not
necessary conditions.

In practice, it is more likely to encounter some unexpected
failures. Once that happens, the performance of control systems is
certainly affected or even irreversible. Therefore, various design
frameworks for reliable controllers have been proposed, in which the
non-fragile control has attracted  a remarkable research interest in
recent years \cite{Kumar,Sakthivel,Yangcyb}.  However,
to the best knowledge of authors, there is no work  addressing the
non-fragile finite-time controller design for LDMFS systems. Up to
now, for LDMFS systems, we can only find few works such as
\cite{Niyuanhua1} on linear quadratic optimal control problem  and
\cite{zhangchen}  about cooperative linear quadratic dynamic
difference game.

In this paper, we investigate the   finite-time stabilization
of LDMFS systems via non-fragile control. The basic approach is
based on STMs  of LDMFS systems. In \cite{IJC2016}, STMs of linear
discrete stochastic systems were firstly presented and employed to
investigate the exact observability, exact/uniform detectability and
Lyapunov-type theorems of the following classical stochastic system
with multiplicative noises:
\begin{equation}\label{2019finite-system7}
\begin{cases}
x_k=H_kx_k+M_kx_kw_k,\\
y_k=G_kx_k.
\end{cases}
\end{equation}
In \cite{ziji-finite}, the method of STM was first used to discuss the
finite-time stability  of system (\ref{2019finite-system7}).
However, this method has not been applied to LDMFS  systems, this is
because that, in LDMFS systems, it is very difficult to establish
the STM expressions. The contributions of this paper are highlighted
as follows:
\begin{itemize}

  \item     Some specific expression forms of STMs
  have been established
  by iterative equations. Based on the  linear transformation and  the  augmented  system  method,
  we establish an equivalent relationship between the original LDMFS system
  with uncertain parameters and a certain augmented non-mean-field time-varying discrete stochastic system with random coefficients.

  \item
 Based on the STM approach, several necessary and sufficient conditions for  the
  finite-time stabilization of the LDTMF system have been obtained.

  \item With the increase of the length of the time interval of interest, the criteria  obtained by STM on finite-time
  stabilization  often leads to  higher computational
  complexity. To reduce the computational complexity, we construct  novel necessary and sufficient Lyapunov-type conditions by
  using the  introduced  STMs, and obtain a sufficient condition to guarantee the  finite-time
  stabilization in the form of linear matrix inequalities (LMIs), which is easier to use in designing the finite-time controller.
\end{itemize}

The rest of this paper is organized as follows: In Section  II, some
useful definitions and lemmas are introduced. In Section
 III, we investigate the STM approach of LDMFS systems
and its application to finite-time stabilization. By system
reconfiguration, we transform the original system  into a new
discrete stochastic system with state dependent noise. Necessary and
sufficient conditions are presented to solve the stabilization
problem. One example is given in Section   IV  to illustrate the
effectiveness of the theoretic results obtained. The conclusion is
drawn in  Section  VI.

For convenience, we present the notations used in this article here:
${\mathcal R}^n$ denotes the $n$-dimensional real Euclidean vector
space and ${\mathcal R}^{n\times m}$ stands for the space of all
$n\times m$ real matrices.  $\|\cdot\|$ means the Euclidean norm.
The notation $C>0$ means that the matrix $C$ is positive definite
real symmetric and $C<0$ means that the matrix $C$ is negative
definite real symmetric. $C'$ stands for the transpose of the matrix
or vector $C$.  $I_m$ denotes the $m\times m$ identity matrix. Given
matrices $F$ and $G$, the notation $F\otimes G$ stands for the
Kronecker product of $F$ and $G$. Given a positive integer $M$,
  ${\mathcal N}_M$ means the set $\{0,1, 2,\cdots, M\}$ and
  $diag(a_1,a_2,\cdots,a_m)$ means
  a diagonal matrix whose leading diagonal
  entries  are  $a_1,a_2,\cdots,a_m$. The notation $\mathcal {E}$ denotes the mathematical expectation
 operator.
\section{Preliminaries\label{sec:PR}}

\hspace{0.13in}In this section, we will consider the following LDMFS system
\begin{equation}\label{2019finite-system1}
\begin{cases}
x_{k+1}=A_1x_k+A_2\mathcal{E}x_k+Bu^F_k\\
\ \ \ \ \ \ \ \ \ \ +(C_1x_k+C_2\mathcal {E}x_k+Du^F_k)w_k,\\
x_0=\xi\in {\mathcal R}^n, \ k\in {\mathcal N}_{T-1},
\end{cases}
\end{equation}
where $x_k\in{\mathcal R}^n$ and $u^F_k\in{\mathcal R}^m$  are  the
state vector and actuator output vector with fault at time $k$,
respectively. Suppose that the initial state $x_0$ is a deterministic real  vector $\xi$. $\{w_k\}_{k\in {\mathcal N}_{T-1}}$ stands for the
system noise assumed to be a one-dimensional independent white
noise sequence defined on the  probability space $(\Omega,
{\mathcal F}, {\mathcal P})$.    Assume that $\mathcal {E}[w_k]=0$,
$\mathcal {E}[w_iw_k]=0$ when $i\neq k$, and $\mathcal
{E}[w_iw_k]=1$ when $i=k$. $A_1$, $A_2$, $B$, $C_1$, $C_2$ and $D$
are deterministic matrices with appropriate dimensions. In  most
practical systems, the uncertain parameters in feedback coefficients
usually cannot be ignored due to that the state feedback control may
be extremely sensitive or fragile with respect to errors. So we have
to consider  the following non-fragile state feedback control
\begin{eqnarray}\label{2019finite-u1}
u^F_k=(K_1+\alpha_k\Delta K_{1,k})x_k+(K_2+\alpha_k\Delta K_{2,k})\mathcal {E}x_k,
\end{eqnarray}
where $K_1$ and $K_2$ are the control gain matrices. $\Delta
K_{1,k}$ and $\Delta K_{2,k}$ are the uncertain parameters
satisfying $[\Delta K_{1,k}\ \Delta K_{2,k}]=MF_k[N_1\ N_2]$, where
$M$, $N_1$ and $N_2$   are   known matrices of appropriate dimensions,
and  $F_k$ is the uncertain  matrix  with $F'_kF_k\leq I$.
$\{\alpha_k\}_{k\in{\mathcal N}_{T-1}}$, which is independent of
 $\{w_k\}_{k\in{\mathcal N}_{T-1}}$,  is a
  finite sequence of independent
 random variables
  satisfying Bernoulli distribution with $\mathcal {P}(\alpha_k=1)=\bar{\alpha}$
  and $\mathcal {P}(\alpha_k=0)=1-\bar{\alpha}$, $0\leq\bar{\alpha}\leq1$.

\begin{definition}\label{def1}
Give any positive integer $T>0$, two positive numbers $\epsilon_1$ and $\epsilon_2$ with $0<\epsilon_1\le \epsilon_2$, and a sequence of positive definite symmetric matrices $\{R_k\}_{k\in{\mathcal N}_{T}}$.
If there exists a non-fragile
controller $u^F_k$ such that the following closed-loop system
{\small
\begin{eqnarray}\label{2019finite-system2}
\begin{cases}
x_{k+1}=(A_1+BK_1+\alpha_kB\Delta K_{1,k})x_k+(A_2+BK_2\\
\ \ \ \ \ \ \ \ \ +\alpha_kB\Delta K_{2,k})\mathcal {E}x_k+[(C_1+DK_1\\
\ \ \ \ \ \ \ \ \ +\alpha_kD\Delta K_{1,k})x_k+(C_2+DK_2+\alpha_kD\Delta K_{2,k})\\
\ \ \ \ \ \ \ \ \ \mathcal {E}x_k]w_k,\\
x_0=\xi, \  k\in {\mathcal N}_{T-1}
\end{cases}
\end{eqnarray}}
satisfies
\begin{equation}\label{2019finite-const}
x_0'R_0x_0\leq \epsilon_1\Rightarrow \mathcal {E}(x_k'R_kx_k)\leq \epsilon_2, \ \ \
\forall k \in {\mathcal N}_{T},
\end{equation}
then system (\ref{2019finite-system1}) is said to be finite-time
stabilizable with respect to $(\epsilon_1, \epsilon_2, T, \{R_k\}_{k\in{\mathcal N}_{T}})$.
\end{definition}

The following property will be used to prove our main results.

\begin{lemma}\label{2019finite-Lem2}\cite{ziji-finite}
For given matrices $F$, $G$, $H$ and $M$ of suitable dimensions, the following holds:
\begin{equation}
(F\otimes G)(H\otimes M)=(FH)\otimes(GM).
\end{equation}
\end{lemma}

\section{State Transition Matrix and Finite-time Stabilization}

In this section, we will firstly build the STM of LDMFS system (\ref{2019finite-system1}) and then research the finite-time stabilization of LDMFS system (\ref{2019finite-system1}) based on the STM approach. With that
$\alpha_k$ and $x_k$ are independent of each other, taking the
mathematical expectation  in system (\ref{2019finite-system2}), it
follows that
\begin{align}\label{2019finite-system6}
\begin{cases}
\mathcal {E} x_{k+1}=[A_1+A_2+B(K_1+K_2)\\
\ \ \ \ \ \ \ \ \ \ \ +\bar{\alpha}B(\Delta K_{1,k}+\Delta K_{2,k})]\mathcal {E} x_k,\\
\mathcal {E} x_{0}=x_0=\xi,\ k\in {\mathcal N}_{T-1}.
\end{cases}
\end{align}
Subtracting (\ref{2019finite-system6}) from
(\ref{2019finite-system2}) and setting $\hat{x}_k=x_k-{\mathcal {E}
x}_k$, we have

\begin{eqnarray*}
\begin{cases}
\hat{x}_{k+1}=(A_1+BK_1+\alpha_kB\Delta K_{1,k})\hat{x}_{k}+[\alpha_kB(\Delta K_{1,k}\\
\ \ \ \ \ \ \ \ \ +\Delta K_{2,k})-\bar{\alpha}B(\Delta K_{1,k}+\Delta K_{2,k})]\mathcal{E}x_k\\
\ \ \ \ \ \ \ \ \ +[(C_1+DK_1+\alpha_kD\Delta K_{1,k})\hat{x}_{k}\\
\ \ \ \ \ \ \ \ \ +(C_2+C_1+DK_2+DK_1+\alpha_kD\Delta K_{2,k}\\
\ \ \ \ \ \ \ \ \ +\alpha_kD\Delta K_{1,k})\mathcal {E}x_k]w_k,\\
\hat{x}_{0}=0,\ k\in {\mathcal N}_{T-1}.
\end{cases}
\end{eqnarray*}
Letting $\tilde{x}_k=\left[\begin{array}{ccc}{\mathcal {E} x}_k\\
\hat{x}_k\end{array}\right]$, we can obtain the following augmented
system with respect to  $\tilde{x}_k$:
\begin{align}\label{2019finite-oldsystem4}
\begin{cases}
\tilde{x}_{k+1}=\tilde{A}_k\tilde{x}_k+\tilde{C}_k\tilde{x}_kw_k,\\
\tilde{x}_{0}=\left[\begin{array}{ccc}\xi\\0\end{array}\right],\ k\in {\mathcal N}_{T-1},
\end{cases}
\end{align}
where
{\small \begin{align*}
\tilde{A}_k=&\left[\begin{array}{ccc}A_1+A_2+B(K_1+K_2)+\bar{\alpha}B(\Delta K_{1,k}+\Delta K_{2,k})\\
\alpha_kB(\Delta K_{1,k}+\Delta K_{2,k})-\bar{\alpha}B(\Delta K_{1,k}+\Delta K_{2,k})
\end{array}\right.\\
&\left.\begin{array}{ccc}0\\ A_1+BK_1+\alpha_kB\Delta K_{1,k}\end{array}\right],\\
\tilde{C}_k=&\left[\begin{array}{ccc}
0\\
C_2+C_1+D(K_2+K_1)+\alpha_kD(\Delta K_{2,k}+\Delta K_{1,k})
\end{array}\right.\\
&\left.\begin{array}{ccc}0\\ C_1+DK_1+\alpha_kD\Delta K_{1,k}\end{array}\right].
\end{align*}}
Note that $\mathcal {E}\|\tilde{x}_k\|^2=(\mathcal {E}
x'_k)(\mathcal {E} x_k)+\mathcal {E}(\hat{x}'_k\hat{x}_k)= (\mathcal
{E} x'_k)(\mathcal {E} x_k)+\mathcal {E}[(x_k-\mathcal {E}
x_k)'(x_k-\mathcal {E} x_k)]=\mathcal {E}\|x_k\|^2$, $k\in {\mathcal
N}_{T}$. To study the second-order moment  of $\tilde{x}_k$ in
system (\ref{2019finite-oldsystem4}), we  need to  prove some
lemmas.

\begin{remark}\label{2019finite-rem1}
If we set
{\small
\begin{align*}
\mathcal {A}_{1,k}=&\left[\begin{array}{ccc}A_1+A_2+B(K_1+K_2)+\bar{\alpha}B(\Delta K_{1,k}+\Delta K_{2,k})\\
-\bar{\alpha}B(\Delta K_{1,k}+\Delta K_{2,k})
\end{array}\right.\\
&\left.\begin{array}{ccc}0\\ A_1+BK_1\end{array}\right],\\
\mathcal {A}_{2,k}=&\left[\begin{array}{ccc}0&0\\
B(\Delta K_{1,k}+\Delta K_{2,k})&B\Delta K_{1,k}
\end{array}\right],\\
\mathcal {C}_{1}=&\left[\begin{array}{ccc}
0&0\\
C_2+C_1+D(K_2+K_1)&C_1+DK_1
\end{array}\right]
\end{align*}}
and
$$
\mathcal {C}_{2,k}=\left[\begin{array}{ccc}
0&0\\
D(\Delta K_{2,k}+\Delta K_{1,k})&D\Delta K_{1,k}
\end{array}\right],
$$
then $\tilde{A}_k=\mathcal {A}_{1,k}+\alpha_k\mathcal {A}_{2,k}$, $\tilde{C}_k=\mathcal {C}_{1}+\alpha_k\mathcal {C}_{2,k}$ and system (\ref{2019finite-oldsystem4}) can be rewritten into
\begin{eqnarray}\label{2019finite-system4}
\begin{cases}
\tilde{x}_{k+1}=(\mathcal {A}_{1,k}+\alpha_k\mathcal {A}_{2,k})\tilde{x}_{k}+(\mathcal {C}_{1}+\alpha_k\mathcal {C}_{2,k})\tilde{x}_{k}w_k,\\
\tilde{x}_{0}=\left[\begin{array}{ccc}\xi\\0\end{array}\right],\ k\in {\mathcal N}_{T-1}.
\end{cases}
\end{eqnarray}
\end{remark}
Denote
\begin{align*}
&\psi_{l,k}:=\left[
\begin{array}{cccc}
\psi_{l,k+1}(\sqrt{\bar{\alpha}}\mathcal {A}_{1,k}+\sqrt{\bar{\alpha}}\mathcal {A}_{2,k})\\
\psi_{l,k+1}\sqrt{1-\bar{\alpha}}\mathcal {A}_{1,k}\\
\psi_{l,k+1}(\sqrt{\bar{\alpha}}\mathcal {C}_{1}+\sqrt{\bar{\alpha}}\mathcal {C}_{2,k})\\
\psi_{l,k+1}\sqrt{1-\bar{\alpha}}\mathcal {C}_{1}\\
\end{array}
\right],
\end{align*}
$l>k, \psi_{k,k}=I_{2n}, \forall k\in {\mathcal N}_{T}$.
In addition, there exists another expression for $\psi_{l,k}$ that is denoted by $\varphi_{l,k}$ in the following lemma. These two expressions are both needed in the  proof process of our subsequent results.

\begin{lemma}\label{2019finite-lemma2d}
Set
\begin{equation*}
\varphi_{l,k}=\left(I_{4^{l-k-1}}\otimes \left[\begin{array}{cccc}
\sqrt{\bar{\alpha}}\mathcal {A}_{1,l-1}+\sqrt{\bar{\alpha}}\mathcal {A}_{2,l-1}\\
\sqrt{1-\bar{\alpha}}\mathcal {A}_{1,l-1}\\
\sqrt{\bar{\alpha}}\mathcal {C}_{1}+\sqrt{\bar{\alpha}}\mathcal {C}_{2,l-1}\\
\sqrt{1-\bar{\alpha}}\mathcal {C}_{1}
\end{array}
\right]\right)\varphi_{l-1,k},
\end{equation*}
$l>k; \varphi_{k,k}=I_{2n}, \forall k\in {\mathcal
N}_{T}$. Then, we have the following relation:
\begin{equation}\label{eqvghvh}
\varphi_{l,k}=\psi_{l,k}, \ \forall l>k\in\mathcal {N}_T.
\end{equation}
\end{lemma}
\textbf {Proof.} The proof can be found in APPENDIX.

\begin{remark}
  By Lemma~\ref{2019finite-lemma2d}, the matrices
$\psi_{\cdot,\cdot}$ and $\varphi_{\cdot,\cdot}$ actually represent
the same matrix, but the difference lies in different iterative
expressions. $\varphi_{j,i}$ is calculated  in forward time, while
$\psi_{j,i}$ is calculated  in backward time. Therefore,
$\varphi_{j,i}$ is  in line with the characteristics of the STM   of
deterministic linear discrete systems.      The introduction of
$\varphi_{j,i}$ is one important contribution of this paper.
Lemma~\ref{2019finite-lemma2d} is new even in non-mean-field
stochastic systems, and plays an important role in this  paper.
\end{remark}

In the following, we uniformly denote  $\psi_{\cdot,\cdot}$ and
$\varphi_{\cdot,\cdot}$  as  $\psi_{\cdot,\cdot}$ for simplicity.
Moreover, we define another matrix $\phi_{l,k}$ as {\small
\begin{eqnarray}\label{2019finite-as12}
\phi_{l,k}=\left[
\begin{array}{cccc}
[I_{4^{l-k-1}}\otimes (\sqrt{\bar{\alpha}}\mathcal {A}_{1,l-1}+\sqrt{\bar{\alpha}}\mathcal {A}_{2,l-1})]\phi_{l-1,k}\\
(I_{4^{l-k-1}}\otimes \sqrt{1-\bar{\alpha}}\mathcal {A}_{1,l-1})\phi_{l-1,k}\\
(I_{4^{l-k-1}}\otimes (\sqrt{\bar{\alpha}}\mathcal {C}_{1}+\sqrt{\bar{\alpha}}\mathcal {C}_{2,l-1}))\phi_{l-1,k}\\
(I_{4^{l-k-1}}\otimes \sqrt{1-\bar{\alpha}}\mathcal
{C}_{1})\phi_{l-1,k}
\end{array}
\right],
\end{eqnarray}}
$l>k, \phi_{k,k}=I_{2n}, \forall k\in {\mathcal N}_T$. On the basis of Lemma~\ref{2019finite-lemma2d}, we further give the
following result for the state transition of
system (\ref{2019finite-system4}) in mean square sense.

\begin{lemma}\label{2019finite-lemma2}
For system (\ref{2019finite-system4}), we have  the following iterative
 relations:
\begin{align}\label{2019finite-as13}
\mathcal
{E}\|\tilde{x}_l\|^2=\mathcal
{E}\|\psi_{l,k}\tilde{x}_k\|^2, \ l\ge k,
\end{align}
where $\psi_{l,k}$ is defined as in  Lemma~\ref{2019finite-lemma2d}.
\begin{align}\label{2019finite-asa13}
\mathcal {E}\|\tilde{x}_l\|^2= \mathcal
{E}\|\phi_{l,k}\tilde{x}_k\|^2, \ l\ge k,
\end{align} where $\phi_{l,k}$ is defined as in
(\ref{2019finite-as12}).
\end{lemma}

\textbf {Proof.} The proof can be found in APPENDIX.

\begin{remark}
The matrices $\phi_{\cdot,\cdot}$ and $\psi_{\cdot,\cdot}$ can be
regarded as the STMs in
 mean square sense of discrete stochastic system (\ref{2019finite-system4}) with random
 coefficients.   Different from deterministic systems,  STMs are  not
 unique in discrete stochastic systems, which have several expression forms.
\end{remark}

We are now in a  position to  make  the  connection between the
finite-time stabilization of mean-field system
(\ref{2019finite-system1}) and another classical time-varying
stochastic system. Set $\bar {R}_k=diag(R_k, R_k)$ and
$\bar{x}_k=\bar {R}^{\frac{1}{2}}_k\tilde{x}_k$, then we have the
following lemma.

\begin{lemma}\label{2019finite-lemma5}
System (\ref{2019finite-system1}) is finite-time stabilizable with
respect to $(\epsilon_1, \epsilon_2, T, \{R_k\}_{k\in{\mathcal N}_{T}})$  if and only if (iff)
the system
\begin{eqnarray}\label{2019finite-system5}
\begin{cases}
\bar{x}_{k+1}=(\bar{\mathcal {A}}_{1,k}+\alpha_k\bar{\mathcal {A}}_{2,k})\bar{x}_k
+(\bar{\mathcal {C}}_{1,k}+\alpha_k\bar{\mathcal {C}}_{2,k})\bar{x}_kw_k,\\
\bar{x}_{0}=\left[\begin{array}{ccc}
R_0^{\frac{1}{2}}\xi\\0\end{array}\right], k\in {\mathcal N}_{T-1}
\end{cases}
\end{eqnarray}
is finite-time stable with respect to $(\epsilon_1, \epsilon_2, T, I_{2n})$, where
\begin{align*}
\bar{\mathcal {A}}_{1,k}=\bar
{R}_{k+1}^{\frac{1}{2}}\mathcal {A}_{1,k}\bar{R}_k^{-\frac{1}{2}},\
\bar{\mathcal {A}}_{2,k}=\bar{R}_{k+1}^{\frac{1}{2}}\mathcal
{A}_{2,k}\bar {R}_k^{-\frac{1}{2}},\\
 \bar{\mathcal {C}}_{1,k}=\bar {R}_{k+1}^{\frac{1}{2}}\mathcal {C}_{1}\bar
{R}_k^{-\frac{1}{2}},\ \bar{\mathcal {C}}_{2,k}=\bar
{R}_{k+1}^{\frac{1}{2}}\mathcal {C}_{2,k}\bar{R}_k^{-\frac{1}{2}}.
\end{align*}
Moreover, the corresponding STMs $\bar{\psi}_{l,k}$ and $\bar{\phi}_{l,k}$ are  given by
\begin{eqnarray*}
\begin{cases}
\bar{\phi}_{l,k}=\left[
\begin{array}{cccc}
(I_{4^{l-k-1}}\otimes (\sqrt{\bar{\alpha}}\bar{\mathcal {A}}_{1,l-1}+\sqrt{\bar{\alpha}}\bar{\mathcal {A}}_{2,l-1}))\bar{\phi}_{l-1,k}\\
(I_{4^{l-k-1}}\otimes \sqrt{1-\bar{\alpha}}\bar{\mathcal {A}}_{1,l-1})\bar{\phi}_{l-1,k}\\
(I_{4^{l-k-1}}\otimes (\sqrt{\bar{\alpha}}\bar{\mathcal {C}}_{1,k}+\sqrt{\bar{\alpha}}\bar{\mathcal {C}}_{2,l-1}))\bar{\phi}_{l-1,k}\\
(I_{4^{l-k-1}}\otimes \sqrt{1-\bar{\alpha}}\bar{\mathcal
{C}}_{1,k})\bar{\phi}_{l-1,k}
\end{array}
\right],\\
\bar{\phi}_{k,k}=I_{2n},
\end{cases}
\end{eqnarray*}
and
\begin{eqnarray*}
\begin{cases}
\bar{\psi}_{l,k}=\left[
\begin{array}{cccc}
\bar{\psi}_{l,k+1}(\sqrt{\bar{\alpha}}\bar{\mathcal {A}}_{1,k}+\sqrt{\bar{\alpha}}\bar{\mathcal {A}}_{2,k})\\
\bar{\psi}_{l,k+1}\sqrt{1-\bar{\alpha}}\bar{\mathcal {A}}_{1,k}\\
\bar{\psi}_{l,k+1}(\sqrt{\bar{\alpha}}\bar{\mathcal {C}}_{1,k}+\sqrt{\bar{\alpha}}\bar{\mathcal {C}}_{2,k})\\
\bar{\psi}_{l,k+1}\sqrt{1-\bar{\alpha}}\bar{\mathcal {C}}_{1,k}\\
\end{array}
\right],\\
\bar{\psi}_{k,k}=I_{2n},
\end{cases}
\end{eqnarray*}
respectively.
\end{lemma}

{\textbf {Proof.}} The proof can be found in APPENDIX.

The next two lemmas  are dedicated to finding the relationship between  $\phi_{l,k}$ and  $\bar{\phi}_{l,k}$, and  $\psi_{l,k}$ and
$\bar{\psi}_{l,k}$, respectively.

\begin{lemma}\label{2019finite-Lem1}
For any $0\le k\le l$, assume that the matrices $\phi_{l,k}$ and $\bar{\phi}_{l,k}$ are
STMs of systems
(\ref{2019finite-system4}) and (\ref{2019finite-system5}), respectively. Then the
following relation always holds:
\begin{align*}
\phi'_{l,k}(I_{4^{l-k}}\otimes \bar{R}_l)\phi_{l,k}=\bar
{R}_k^{\frac{1}{2}}\bar{\phi}'_{l,k}\bar{\phi}_{l,k}\bar
{R}_k^{\frac{1}{2}}.
\end{align*}
\end{lemma}
{\textbf {Proof.}} The proof can be found in APPENDIX.

\begin{lemma}\label{2019finite-Lem3}
For any $0\le k\le l$, assume that the matrices $\psi_{l,k}$ and $\bar{\psi}_{l,k}$ are
STMs of systems
(\ref{2019finite-system4}) and (\ref{2019finite-system5}), respectively. Then the
following relation always holds:
\begin{align}\label{2019finite-gfvccc}
\psi'_{l,k}(I_{4^{l-k}}\otimes \bar{R}_l)\psi_{l,k}=\bar
{R}_k^{\frac{1}{2}}\bar{\psi}'_{l,k}\bar{\psi}_{l,k}\bar
{R}_k^{\frac{1}{2}}.
\end{align}
\end{lemma}
{\textbf {Proof.}} The proof can be found in APPENDIX.

\begin{theorem}\label{2019finite-th1}
For an integer $T>0$, two positive scalars $\epsilon_1$ and $\epsilon_2$ with
$0<\epsilon_1\le \epsilon_2$, and  a sequence of positive definite symmetric matrices $\{R_k\}_{k\in{\mathcal N}_{T}}$, the
following conditions are equivalent:
\begin{description}
  \item[(a)] LDMFS system (\ref{2019finite-system1}) is finite-time stabilizable with respect to $(\epsilon_1, \epsilon_2, T, \{R_k\}_{k\in{\mathcal N}_{T}})$.
  \item[(b)]
   \begin{equation} \label{2019finite-as1}
\phi_{k,0}'(I_{4^k}\otimes \bar
{R}_k)\phi_{k,0}\leq\frac{\epsilon_2}{\epsilon_1}\bar{R}_0, \ \ \
\forall k\in {\mathcal N}_T.
\end{equation}

\item[(c)]

\begin{equation}\label{2019finite-as3}
\bar{\phi}_{k,0}'\bar{\phi}_{k,0}\leq\frac{\epsilon_2}{\epsilon_1}I_{2n},\ \ \
\forall k\in {\mathcal N}_T.
\end{equation}

\item[(d)]

   \begin{equation}\label{2019finite-as4}
\psi_{k,0}'(I_{4^k}\otimes \bar{R}_k)\psi_{k,0}\leq\frac{\epsilon_2}{\epsilon_1}\bar{R}_0, \ \ \
\forall k\in {\mathcal N}_T.
\end{equation}

\item[(e)]

\begin{equation}\label{2019finite-as5}
\bar{\psi}_{k,0}'\bar{\psi}_{k,0}\leq\frac{\epsilon_2}{\epsilon_1}I_{2n},\ \ \
\forall k\in {\mathcal N}_T.
\end{equation}

  \item[(f)] There are  symmetric  matrices  $P_k, k\in \mathcal {N}_T, $ such that the
following constrained difference equation holds:
\begin{eqnarray*}
\begin{cases}
P_{0}= \bar{R}_0^{-1}, \\
P_{k+1}=\left[\begin{array}{cccccc}
I_{4^{k}}\otimes (\sqrt{\bar{\alpha}}\mathcal {A}_{1,k}+\sqrt{\bar{\alpha}}\mathcal {A}_{2,k})\\
I_{4^{k}}\otimes \sqrt{1-\bar{\alpha}}\mathcal {A}_{1,k}\\
I_{4^{k}}\otimes (\sqrt{\bar{\alpha}}\mathcal {C}_{1}+\sqrt{\bar{\alpha}}\mathcal {C}_{2,k})\\
I_{4^{k}}\otimes \sqrt{1-\bar{\alpha}}\mathcal {C}_{1}
\end{array}\right]P_k\\
\ \ \ \ \ \ \ \ \ \ \ \left[\begin{array}{cccccc}
I_{4^{k}}\otimes (\sqrt{\bar{\alpha}}\mathcal {A}_{1,k}+\sqrt{\bar{\alpha}}\mathcal {A}_{2,k})\\
I_{4^{k}}\otimes \sqrt{1-\bar{\alpha}}\mathcal {A}_{1,k}\\
I_{4^{k}}\otimes (\sqrt{\bar{\alpha}}\mathcal {C}_{1}+\sqrt{\bar{\alpha}}\mathcal {C}_{2,k})\\
I_{4^{k}}\otimes \sqrt{1-\bar{\alpha}}\mathcal {C}_{1}
\end{array}\right]',\\
P_k\leq\frac{\epsilon_2}{\epsilon_1}(I_{4^k}\otimes \bar{R}_k^{-1}),\ k\in
{\mathcal N}_T.
\end{cases}
\end{eqnarray*}
\end{description}
\end{theorem}
{\textbf {Proof.}} The proof can be found in APPENDIX.

\begin{remark}
The necessary and sufficient conditions for finite-time
stabilizability  of LDMFS system  (\ref{2019finite-system1}) are
 presented in Theorem \ref{2019finite-th1}. When $u_k^F=0$
for $k\in {\mathcal N}_{T-1}$, then  necessary and sufficient
conditions for finite-time stability  of  the  following  unforced
system
$$
\begin{cases}
x_{k+1}=A_1x_k+A_2\mathcal{E}x_k\\
\ \ \ \ \ \ \ \ \ \ +(C_1x_k+C_2\mathcal {E}x_k)w_k,\\
x_0=\xi\in {\mathcal R}^n, \ k\in {\mathcal N}_{T-1},
\end{cases}
$$
are given.  When  system (\ref{2019finite-system1}) degenerates into
a standard linear discrete stochastic system without mean-field
terms, similar  results first appeared in \cite{ziji-finite}. A main
difficulty  to give necessary and sufficient conditions for
finite-time stability and stabilizability of LDMFS systems exists in
that it is not easy to  obtain  the STMs as seen above, which
differs  from linear deterministic systems
\cite{Amato2,Amato3,Amato}. In \cite{Amato2,Amato3,Amato}, necessary
and sufficient conditions have been given for finite-time stability
of  linear deterministic systems  based on  the  STM.
\end{remark}

\section{Construction of Lyapunov function based on STMs}

In Theorem \ref{2019finite-th1}, five criteria are given through
STMs. These criteria are all necessary and sufficient conditions for
finite-time stabilization, and the first four criteria are
relatively simple in form. However, solving these inequalities in
Theorem \ref{2019finite-th1} is not easy when $T$ is large enough.
For example, when using {\rm (f)} in Theorem \ref{2019finite-th1} to
verify the finite-time stabilization of LDMFS system
(\ref{2019finite-system1}), with the progressive increase  of  $k$,
the order of the solution matrix $P_k$ keeps expanding and is
$2^{2k+1}n\times 2^{2k+1}n$. Next, we will find ways to simplify the
calculation of Theorem \ref{2019finite-th1} and find a novel
Lyapunov-type theorem.

Let  $\Gamma_{r}$ denote the set of block matrices composed of ~$r\times r$ square sub-matrices with the same dimension. For   the block matrix $A$ belongs to $\Gamma_{r}$ with $A_{ij}$ denoting   its sub-matrix,  we introduce an operator  $\text{Tr}(A[A_{ij}])=\sum^r_{i=1}A_{ii}$. As a generalization of the standard  matrix trace, $\text{Tr}$ can be called block trace. It is not difficult to find that $\text{Tr}$ enjoys the following useful properties.

\begin{lemma}\label{2019finite-Lem6}
For any block matrix $A[A_{ij}]_{r\times r}\in\Gamma_{r}$, the following are true:
\begin{itemize}
\item[(i)] $\text{Tr}(A')=\text{Tr}(A)'$.
\item[(ii)] For any $2n$ matrices $C_i$, $D_i$ with appropriate dimension, $i\in\{1,2,\cdots,n\}$, there will always be
    \begin{align*}&\text{Tr}\left(\left(I_r\otimes \left[\begin{array}{cccc}C_1\\\vdots\\C_n\end{array}\right]\right)A\left(I_r\otimes \left[\begin{array}{cccc}D_1\\\vdots\\D_n\end{array}\right]\right)'\right)
    \\=&\sum^n_{i=1}C_i\text{Tr}(A)D_i'.
    \end{align*}
\end{itemize}
\end{lemma}

\begin{proof}
(i) is obvious, so we only need to show  (ii). Without loss of generality,
set
$$\Theta[\Theta_{ij}]=\left(I_r\otimes \left[\begin{array}{cccc}C_1\\\vdots\\C_n\end{array}\right]\right)A\left(I_r\otimes \left[\begin{array}{cccc}D_1\\\vdots\\D_n\end{array}\right]\right)',$$
then $\Theta[\Theta_{ij}]$ is a block matrix with $1\le i,j\leq r\times n$.
$\lfloor\cdot\rfloor$ stands for the floor function, i.e., $\lfloor\alpha\rfloor=\{\max \beta\in N | \beta\leq\alpha\}$. Meanwhile, $\lceil\cdot\rceil$ is the ceil function, i.e., $\lceil\alpha\rceil=\{\max \beta\in N | \beta\geq\alpha\}$.
When $\nu_i=i-\lfloor\frac{i}{n}\rfloor\times n$ and $\mu_s=\lceil\frac{s}{n}\rceil$,
we have $\Theta_{ij}=C_{I_{\{\nu_i=0\}}\times r+\nu_i}A_{\mu_i\mu_j}D'_{I_{\{\nu_j=0\}}\times r+\nu_j}$.
Therefore,
\begin{align*}
\text{Tr}(\Theta[\Theta_{ij}])=\sum^{r\times n}_{i=1}D_{ii}=\sum^{r}_{i=1}\sum^{n}_{j=1}C_jA_{ii}D_j'=\sum^n_{j=1}C_j\text{Tr}(A)D_j'.
\end{align*}
The proof is completed.
\end{proof}

\begin{remark}
The properties of $\text{Tr}$ and the standard matrix trace $\text{tr}$ are not completely consistent, such as commutativity.
Generally speaking, $\text{Tr}(AB)=\text{Tr}(BA)$ does not hold.
\end{remark}

Based on Lemma \ref{2019finite-lemma2d} and Theorem \ref{2019finite-th1},
$\bar{\varphi}_{k,0}'\bar{\varphi}_{k,0}=\bar{\psi}_{k,0}'\bar{\psi}_{k,0}\leq\frac{\epsilon_2}{\epsilon_1}I_{2n}$ is a
necessary and sufficient  condition for finite-time stabilization, where
\begin{equation*}
\bar{\varphi}_{k,0}=\left(I_{4^{k-1}}\otimes \left[\begin{array}{cccc}
\sqrt{\bar{\alpha}}\bar{\mathcal {A}}_{1,k-1}+\sqrt{\bar{\alpha}}\bar{\mathcal {A}}_{2,k-1}\\
\sqrt{1-\bar{\alpha}}\bar{\mathcal {A}}_{1,k-1}\\
\sqrt{\bar{\alpha}}\bar{\mathcal {C}}_{1,k-1}+\sqrt{\bar{\alpha}}\bar{\mathcal {C}}_{2,k-1}\\
\sqrt{1-\bar{\alpha}}\bar{\mathcal {C}}_{1,k-1}
\end{array}
\right]\right)\bar{\varphi}_{k-1,0}.
\end{equation*}
Note that $eig_{max}(\bar{\varphi}_{k,0}'\bar{\varphi}_{k,0})=eig_{max}(\text{Tr}(\bar{\varphi}_{k,0}\bar{\varphi}_{k,0}'))$, where $eig_{max}$ means the maximum eigenvalue and $\bar{\varphi}_{k,0}\bar{\varphi}_{k,0}'$ belongs to $\Gamma_{k}$ and each sub-matrix in $\bar{\varphi}_{k,0}\bar{\varphi}_{k,0}'$ belongs to $\mathcal {R}^{2n\times 2n}$.
Set $\bar{P}_k=\text{Tr}(\bar{\varphi}_{k,0}\bar{\varphi}_{k,0}')$.  So we can get that
\begin{align*}
&\bar{P}_{k+1}\\
=&\text{Tr}(\bar{\varphi}_{k+1,0}\bar{\varphi}_{k+1,0}')\\
=&\text{Tr}\left(\left(I_{4^{k}}\otimes \left[\begin{array}{cccc}
\sqrt{\bar{\alpha}}\bar{\mathcal {A}}_{1,k}+\sqrt{\bar{\alpha}}\bar{\mathcal {A}}_{2,k}\\
\sqrt{1-\bar{\alpha}}\bar{\mathcal {A}}_{1,k}\\
\sqrt{\bar{\alpha}}\bar{\mathcal {C}}_{1,k}+\sqrt{\bar{\alpha}}\bar{\mathcal {C}}_{2,k}\\
\sqrt{1-\bar{\alpha}}\bar{\mathcal {C}}_{1,k}
\end{array}\right.
\right]\right)\\
& \left.\bar{\varphi}_{k,0}\bar{\varphi}_{k,0}'\left(I_{4^{k}}\otimes \left[\begin{array}{cccc}
\sqrt{\bar{\alpha}}\bar{\mathcal {A}}_{1,k}+\sqrt{\bar{\alpha}}\bar{\mathcal {A}}_{2,k}\\
\sqrt{1-\bar{\alpha}}\bar{\mathcal {A}}_{1,k}\\
\sqrt{\bar{\alpha}}\bar{\mathcal {C}}_{1,k}+\sqrt{\bar{\alpha}}\bar{\mathcal {C}}_{2,k}\\
\sqrt{1-\bar{\alpha}}\bar{\mathcal {C}}_{1,k}
\end{array}
\right]\right)'\right)
\end{align*}
From Lemma \ref{2019finite-Lem6}, the above equation means that
\begin{align*}
&\bar{P}_{k+1}\\
=&\text{Tr}\left(\left[\begin{array}{cccc}
\sqrt{\bar{\alpha}}\bar{\mathcal {A}}_{1,k}+\sqrt{\bar{\alpha}}\bar{\mathcal {A}}_{2,k}\\
\sqrt{1-\bar{\alpha}}\bar{\mathcal {A}}_{1,k}\\
\sqrt{\bar{\alpha}}\bar{\mathcal {C}}_{1,k}+\sqrt{\bar{\alpha}}\bar{\mathcal {C}}_{2,k}\\
\sqrt{1-\bar{\alpha}}\bar{\mathcal {C}}_{1,k}
\end{array}
\right]\text{Tr}(\bar{\varphi}_{k,0}\bar{\varphi}_{k,0}')\right.\\
&\left.\left[\begin{array}{cccc}
\sqrt{\bar{\alpha}}\bar{\mathcal {A}}_{1,k}+\sqrt{\bar{\alpha}}\bar{\mathcal {A}}_{2,k}\\
\sqrt{1-\bar{\alpha}}\bar{\mathcal {A}}_{1,k}\\
\sqrt{\bar{\alpha}}\bar{\mathcal {C}}_{1,k}+\sqrt{\bar{\alpha}}\bar{\mathcal {C}}_{2,k}\\
\sqrt{1-\bar{\alpha}}\bar{\mathcal {C}}_{1,k}
\end{array}
\right]'\right)\\
=&\bar{\mathcal {A}}_{1,k}\bar{P}_k\bar{\mathcal {A}}_{1,k}'+\bar{\alpha}\bar{\mathcal {A}}_{1,k}\bar{P}_k\bar{\mathcal {A}}_{2,k}'+\bar{\alpha}\bar{\mathcal {A}}_{2,k}\bar{P}_k\bar{\mathcal {A}}_{1,k}'\\
&+\bar{\alpha}\bar{\mathcal {A}}_{2,k}\bar{P}_k\bar{\mathcal {A}}_{2,k}'+\bar{\mathcal {C}}_{1,k}\bar{P}_k\bar{\mathcal {C}}_{1,k}'+\bar{\alpha}\bar{\mathcal {C}}_{1,k}\bar{P}_k\bar{\mathcal {C}}_{2,k}'\\
&+\bar{\alpha}\bar{\mathcal {C}}_{2,k}\bar{P}_k\bar{\mathcal
{C}}_{1,k}'+\bar{\alpha}\bar{\mathcal {C}}_{2,k}\bar{P}_k\bar{\mathcal
{C}}_{2,k}'.
\end{align*}
To sum up the above discussion, the following theorem can  be   easily  obtained.
\begin{theorem}\label{2019finite-th5}
LDMFS system (\ref{2019finite-system1}) is finite-time stabilizable
with respect to $(\epsilon_1, \epsilon_2, T, \{R_k\}_{k\in{\mathcal
N}_{T}})$ iff there are  symmetric positive definite matrices $\{P_k\}_{k\in \mathcal {N}_T}$ satisfying the following constrained  Lyapunov-type
equation {\small\begin{align}\label{2019finite-Ly1}
\begin{cases}
\bar{\mathcal {A}}_{1,k}\bar{P}_k\bar{\mathcal {A}}_{1,k}'+\bar{\alpha}\bar{\mathcal {A}}_{1,k}\bar{P}_k\bar{\mathcal {A}}_{2,k}'+\bar{\alpha}\bar{\mathcal {A}}_{2,k}\bar{P}_k\bar{\mathcal {A}}_{1,k}'+\bar{\alpha}\bar{\mathcal {A}}_{2,k}\bar{P}_k\bar{\mathcal {A}}_{2,k}'\\
+\bar{\mathcal {C}}_{1,k}\bar{P}_k\bar{\mathcal {C}}_{1,k}'+\bar{\alpha}\bar{\mathcal {C}}_{1,k}\bar{P}_k\bar{\mathcal {C}}_{2,k}'+\bar{\alpha}\bar{\mathcal {C}}_{2,k}\bar{P}_k\bar{\mathcal
{C}}_{1,k}'+\bar{\alpha}\bar{\mathcal {C}}_{2,k}\bar{P}_k\bar{\mathcal
{C}}_{2,k}'\\
=\bar{P}_{k+1},\\
\bar{P}_0=I_{2n},\\
\bar{P}_k\leq\frac{\epsilon_2}{\epsilon_1}I_{2n}.
\end{cases}
\end{align}}
\end{theorem}

\begin{remark}
Theorem \ref{2019finite-th5} provides a more convenient way to determine the finite-time stabilization than directly calculating STMs. (\ref{2019finite-Ly1}) can be regarded as a non-fragile mean-field stochastic version of general Lyapunov equation about finite time stabilization. When the system (\ref{2019finite-system1}) degenerates into a classical deterministic system, (\ref{2019finite-Ly1}) reduces to the corresponding results in \cite{Amato}.
This necessary and sufficient Lyapunov-type theorem is able to improve many existing works.
To the best of the authors' knowledge, this result does not exist  even for standard linear discrete stochastic system  $x_{k+1}=A_kx_k+C_kx_kw_k$.
However, Theorem \ref{2019finite-th5} requires $\bar{P}_{k}$ at each step $k$ calculated by an iterative equation.
At present, it is not convenient to design  the non-fragile controller.
Therefore, we try to change the result into  the form of Lyapunov-type inequality easily solved by LMI technique.
\end{remark}

\begin{theorem}\label{2019finite-th6}
LDMFS system (\ref{2019finite-system1}) is finite-time stabilizable
with respect to $(\epsilon_1, \epsilon_2, T, \{R_k\}_{k\in{\mathcal
N}_{T}})$ iff  there are  symmetric positive definite matrices $\{P_k\}_{k\in \mathcal {N}_T}$ satisfying  the following  Lyapunov-type  inequality
\begin{align}\label{2019finite-Ly2}
\begin{cases}
\bar{\mathcal {A}}_{1,k}P_k\bar{\mathcal {A}}_{1,k}'+\bar{\alpha}\bar{\mathcal {A}}_{1,k}P_k\bar{\mathcal {A}}_{2,k}'+\bar{\alpha}\bar{\mathcal {A}}_{2,k}P_k\bar{\mathcal {A}}_{1,k}'\\+\bar{\alpha}\bar{\mathcal {A}}_{2,k}P_k\bar{\mathcal {A}}_{2,k}'
+\bar{\mathcal {C}}_{1,k}P_k\bar{\mathcal {C}}_{1,k}'+\bar{\alpha}\bar{\mathcal {C}}_{1,k}P_k\bar{\mathcal {C}}_{2,k}'\\
+\bar{\alpha}\bar{\mathcal {C}}_{2,k}P_k\bar{\mathcal
{C}}_{1,k}'+\bar{\alpha}\bar{\mathcal {C}}_{2,k}P_k\bar{\mathcal
{C}}_{2,k}'\leq P_{k+1},\\
P_0\geq I_{2n},\\
P_k\leq\frac{\epsilon_2}{\epsilon_1}I_{2n}.
\end{cases}
\end{align}
\end{theorem}

{\textbf {Proof.}}
According to Theorem \ref{2019finite-th5}, we need to prove that the solvability of (\ref{2019finite-Ly1}) and (\ref{2019finite-Ly2}) is equivalent to each other. Through observation, it is not difficult to find that (\ref{2019finite-Ly1}) can definitely deduce (\ref{2019finite-Ly2}) by choosing $P_k=\bar{P}_k$.

Next, let us consider: (\ref{2019finite-Ly2}) $\rightarrow $ (\ref{2019finite-Ly1}). Suppose there are $P_k$ satisfying (\ref{2019finite-Ly2}). Then $0<H_0=P_0^{-1}\leq I$ must exist and $\bar{P}_0=I=H_0P_0$. By induction, it is assumed that there exists symmetric positive definite matrix $H_j$ makes $\bar{P}_j=I=H_jP_j$. Then we need to prove there exits $H_{j+1}$ such that $\bar{P}_{j+1}=I=H_{j+1}P_{j+1}$. From (\ref{2019finite-Ly1}), denote $M_j$ as
\begin{align*}
M_j=&\bar{\mathcal {A}}_{1,k}P_k\bar{\mathcal {A}}_{1,k}'+\bar{\alpha}\bar{\mathcal {A}}_{1,k}P_k\bar{\mathcal {A}}_{2,k}'+\bar{\alpha}\bar{\mathcal {A}}_{2,k}P_k\bar{\mathcal {A}}_{1,k}'\\
&+\bar{\alpha}\bar{\mathcal {A}}_{2,k}P_k\bar{\mathcal {A}}_{2,k}'+\bar{\mathcal {C}}_{1,k}P_k\bar{\mathcal {C}}_{1,k}'+\bar{\alpha}\bar{\mathcal {C}}_{1,k}P_k\bar{\mathcal {C}}_{2,k}'\\
&+\bar{\alpha}\bar{\mathcal {C}}_{2,k}P_k\bar{\mathcal
{C}}_{1,k}'+\bar{\alpha}\bar{\mathcal {C}}_{2,k}P_k\bar{\mathcal
{C}}_{2,k}'.
\end{align*}
So, $H_{j+1}=M_jP_{j+1}^{-1}$. The proof is ended.
$\square$

Next, we are able to  transform the non-fragile finite-time stabilizable controller $u^F_k$ design problem into a feasible solution problem for a set of LMIs based on Schur's  complement.

\begin{theorem}\label{theorem4.3}
LDMFS system (\ref{2019finite-system1}) is finite-time stabilizable with
respect to $(\epsilon_1, \epsilon_2, T, \{R_k\}_{k\in{\mathcal N}_{T}})$ via a non-fragile controller $u^F_k$,
if for a given positive scalar $\gamma>0$, there exist matrices
$K_1$ and $K_2$, positive definite matrices $\{P_k\}_{k\in{\mathcal N}_{T}}$, $\{Q_k\}_{k\in{\mathcal N}_{T-1}}$ solving the following LMIs.
\begin{align}\label{2019finite-LMI1}
&\left[\begin{array}{cccccccccccccc}
-P_{k+1}&\sqrt{\bar{\alpha}}\Pi_1&\sqrt{1-\bar{\alpha}}\Pi_1&\sqrt{\bar{\alpha}}\bar{\mathcal{C}}_{1,k}
&\sqrt{1-\bar{\alpha}}\bar{\mathcal {C}}_{1,k}\\
*&-\frac{\epsilon_1}{\epsilon_2}I_{2n}&0&0&0\\
*&*&-\frac{\epsilon_1}{\epsilon_2}I_{2n}&0&0\\
*&*&*&-\frac{\epsilon_1}{\epsilon_2}I_{2n}&0\\
*&*&*&*&-\frac{\epsilon_1}{\epsilon_2}I_{2n}\\
*&*&*&*&*\\
*&*&*&*&*
\end{array}\right.\nonumber\\
&\left.\begin{array}{ccccccccc}
\Pi_2&0\\
0&\Pi_3\\
0&\Pi_4\\
0&\Pi_5\\
0&0\\
-\gamma I_2\otimes \frac{\epsilon_1}{\epsilon_2}I_{2n}&0\\
*&-I_2\otimes \frac{\epsilon_1}{\epsilon_2}I_{2n}
\end{array}\right]<0,
\end{align}
where $P_0\geq I_{2n}$, $P_k\leq\frac{\epsilon_2}{\epsilon_1}I_{2n}$,
{\small \begin{align*}
\Pi_1=&\bar{R}_{k+1}^{\frac{1}{2}}\left[\begin{array}{ccc}A_1+A_2+B(K_1+K_2)&0\\
0&A_1+BK_1
\end{array}\right]\bar{R}_k^{-\frac{1}{2}},\\
\Pi_2=&\bar{R}_{k+1}^{\frac{1}{2}}\left[\begin{array}{cccc}\bar{\alpha}^{\frac{3}{2}}BM&0\\
(\sqrt{\bar{\alpha}}-\bar{\alpha})BM&\sqrt{\bar{\alpha}}BM
\end{array}\right.\\
&\left.\begin{array}{cccc}\bar{\alpha}\sqrt{1-\bar{\alpha}}BM&0\\
-\bar{\alpha}\sqrt{1-\bar{\alpha}}BM&\bar{\alpha}BM
\end{array}\right](I_2\otimes\bar{R}_k^{-\frac{1}{2}}),
\end{align*}}
{\small \begin{align*}
\Pi_3=&\bar{R}_{k+1}^{\frac{1}{2}}\left[\begin{array}{cccc}\sqrt{\gamma} N_1'+\sqrt{\gamma} N_2'&0&0&0\\0&\sqrt{\gamma} N_1'&0&0\end{array}\right](I_2\otimes\bar{R}_k^{-\frac{1}{2}}),\\
\Pi_4=&\bar{R}_{k+1}^{\frac{1}{2}}\left[\begin{array}{cccc}0&0&\sqrt{\gamma} N_1'+\sqrt{\gamma} N_2'&0\\0&0&0&0\end{array}\right](I_2\otimes\bar{R}_k^{-\frac{1}{2}}),\\
\Pi_5=&\bar{R}_{k+1}^{\frac{1}{2}}\left[\begin{array}{cccc}0&0&0&\sqrt{\gamma} N_1'+\sqrt{\gamma} N_2'\\0&0&0&\sqrt{\gamma} N_1'\end{array}\right](I_2\otimes\bar{R}_k^{-\frac{1}{2}}).
\end{align*}}
\end{theorem}

{\textbf {Proof.}} By Schur's complement, we have a sufficient
condition from (\ref{2019finite-Ly2}) that
{\small\begin{align}\label{2019finite-gsfvst}
&\left[\begin{array}{cccccccccccc}
-P_{k+1}&\sqrt{\bar{\alpha}}(\bar{\mathcal {A}}_{1,k}+\bar{\mathcal {A}}_{2,k})&\sqrt{1-\bar{\alpha}}\bar{\mathcal {A}}_{1,k}&\sqrt{\bar{\alpha}}(\bar{\mathcal {C}}_{1,k}+\bar{\mathcal {C}}_{2,k})\\
*&-\frac{\epsilon_1}{\epsilon_2}I_{2n}&0&0\\
*&*&-\frac{\epsilon_1}{\epsilon_2}I_{2n}&0\\
*&*&*&-\frac{\epsilon_1}{\epsilon_2}I_{2n}\\
*&*&*&*
\end{array}\right.\nonumber\\
&\left.\begin{array}{cccccccccccc}
\sqrt{1-\bar{\alpha}}\bar{\mathcal {C}}_{1}\\
0\\
0\\
0\\
-\frac{\epsilon_1}{\epsilon_2}I_{2n}
\end{array}\right]<0.
\end{align}}
By Theorem 2.7 in \cite{kmin} and Schur's complement,
(\ref{2019finite-gsfvst}) is equivalent to (\ref{2019finite-LMI1}).
$\square$

\begin{remark}
Using the STM method to study the finite-time stability and
stabilization of linear discrete stochastic systems comes from
\cite{ziji-finite}, which is a main  motivation for this study. The
technical novelties compared with  non-mean-field linear stochastic
systems are the following aspects:

   (1) The coefficient matrices  of  the closed-loop system  (\ref{2019finite-system2}) have  uncertain parameters  $\Delta K_{1,k}$, $\Delta K_{2,k}$ and random variable  $\alpha_k$.
   Therefore, the  STMs   of  (\ref{2019finite-system2}) are  more complex in both mathematical derivations and expression forms  than the
   non-mean-field linear discrete stochastic system without  non-fragile control.

   (2) Based on the new STMs   and Lemma \ref{2019finite-Lem6} about a new
   operator  ``$\text{Tr}$'',  novel necessary and sufficient Lyapunov-type theorems (Theorems \ref{2019finite-th5} and \ref{2019finite-th6})    are  proved.
 Theorems \ref{2019finite-th5} and \ref{2019finite-th6}
   are easier to verify than Theorem~\ref{2019finite-th1} especially for larger $T>0$.
   As a corollary of Theorem~\ref{2019finite-th6}, Theorem~\ref{theorem4.3} presents  an LMI-based  sufficient condition for finite-time stabilization
   with the non-fragile control $u^F_k$, which is more easily verified.
\end{remark}

\vspace{1cm}
\section{Verification Example\label{sec:EX}}

\begin{figure}[!htb]
  \centering
  \includegraphics[width=3in]{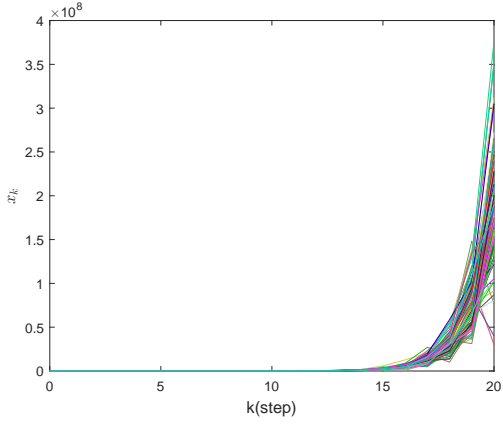}
  \caption{$x_k$ of the open-loop system (\ref{2019finite-eee1}).}
  \label{2019finite-x-open}
\end{figure}

\begin{figure}[!htb]
  \centering
  \includegraphics[width=3in]{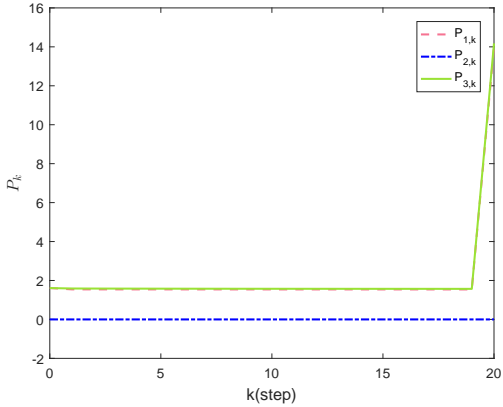}
  \caption{Feasible solutions $P_k$.}
  \label{2019finite-PandQ}
\end{figure}

\begin{figure}[!htb]
  \centering
  \includegraphics[width=3in]{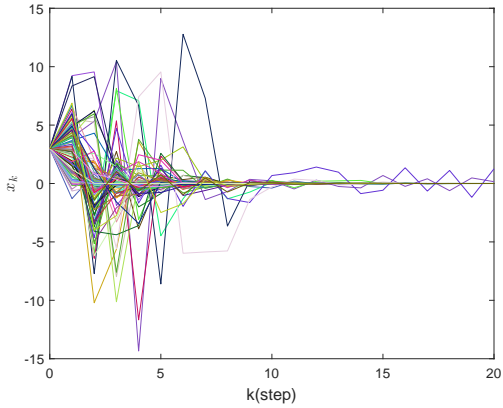}
  \caption{$x_k$ of the closed-loop system (\ref{2019finite-eee1}).}
  \label{2019finite-x-closed}
\end{figure}

\begin{figure}[!htb]
  \centering
  \includegraphics[width=3in]{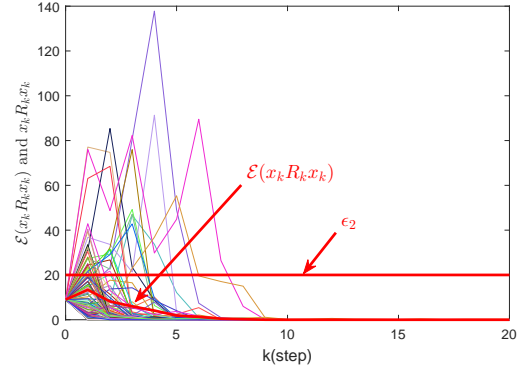}
  \caption{$x_k'Rx_k$ and $\mathcal {E} (x_k'Rx_k)$ of the closed-loop system (\ref{2019finite-eee1}).}
  \label{2019finite-ExRx}
\end{figure}

In the sequel, we present an application to the price control problem of a company's stock market,
which is proposed by \cite{Chen,Linyaning1}. In addition to being adjusted by major shareholders, the stock market is also affected by market fluctuations, changes in crude oil prices, and the average behavior and collective forecast of the stock market.
The deviation between the stock market price and the expected price trajectory is described by the following linear discrete mean-field system
\begin{align}\label{2019finite-eee1}
\begin{cases}
x_{k+1}=&A_1x_k+A_2\mathcal {E}x_k+Bu_k^F\\
&+(C_1x_k+C_2\mathcal {E}x_k+Du_k^F)w_k,
\end{cases}
\end{align}
where the state $x_k$ represents the deviation between the real stock price and the expected value, and the mean term
$\mathcal {E}x_k$ stands for the average impact of collective forecasting on the stock market. The noise $w_k$ can be regarded as the unpredictable influence caused by the continuous fluctuation of unemployment rate and other random events. $u_k$ represents the investment and adjustment of the company's production activities,  but due to the unreasonable organizational structure or internal corruption, the final investment or strategy adjustment is $u_k^F$. If the company expects to keep the deviation within the allowable range within $20$ trading days. We can achieve this goal by the idea of non-fragile finite-time stabilization.

If the parameters of the system (\ref{2019finite-eee1}) are $A_1=1.1833$, $A_2=1.2741$, $B=-1.3517$, $C_1=0.8188$, $C_2=-0.1491$, $D=-0.54$, $M=-0.1005$, $N_1=-0.6177$,
$N_2=0.4285$, and $x_0=3$. Fig. \ref{2019finite-x-open} shows the open-loop state trajectory of $100$ repeated simulations of system (\ref{2019finite-eee1}).
Set $\{R_k=e^{-0.1\times k}\}_{k\in {\mathcal N}_{20}}$, $\epsilon_1=10$, $\epsilon_2=20$, $T=20$, $\gamma=0.5086$. By Theorem \ref{theorem4.3}, we
can obtain the control gain parameters $K_1=0.9627$ and $K_2=0.7737$.
Feasible solutions $P_k=\left[\begin{array}{cc}P_{1,k}&P_{2,k}\\P_{2,k}&P_{3,k}\end{array}\right]$ are presented in Fig. \ref{2019finite-PandQ}.  The corresponding state response of system (\ref{2019finite-eee1}) with control $u_k^F$ is depicted in Fig. \ref{2019finite-x-closed}. The result is confirmed by the time evolution of $\mathcal {E} (x_k'R_kx_k)$ of the closed-loop system (\ref{2019finite-eee1}) presented in Fig. \ref{2019finite-ExRx}.

\section{Conclusion\label{sec:Concl}}
\hspace{0.13in} Several  kinds of STMs of LDMFS systems have been  firstly
presented in this paper and the non-fragile finite-time
stabilization problem has been studied. Based on the STM method, several
necessary and sufficient conditions for non-fragile finite-time stabilization  have been derived. The advantage
of the STM  method  is non-conservative. The feasibility and
effectiveness of the new design schemes have been confirmed by one
example.  We believe that the STM   method will have more applications in   stochastic
stability and  stabilization, which merits further study in our
future work.

\section{Appendix}

\textbf{The proof of Lemma 3.1:}

This lemma can be shown by induction. For any $k$,
if $l=k+1$, (10)  is evidently valid. Suppose that
(10)  holds for  $l=k+m$, i.e.,
{\small
\begin{align}\label{2019finite-as7}
&\left[
\begin{array}{cccc}
\psi_{k+m,k+1}(\sqrt{\bar{\alpha}}\mathcal {A}_{1,k}+\sqrt{\bar{\alpha}}\mathcal {A}_{2,k})\\
\psi_{k+m,k+1}\sqrt{1-\bar{\alpha}}\mathcal {A}_{1,k}\\
\psi_{k+m,k+1}(\sqrt{\bar{\alpha}}\mathcal {C}_{1}+\sqrt{\bar{\alpha}}\mathcal {C}_{2,k})\\
\psi_{k+m,k+1}\sqrt{1-\bar{\alpha}}\mathcal {C}_{1}\\
\end{array}
\right]\nonumber\\
=&\left(I_{4^{m-1}}\otimes\left[
\begin{array}{cccc}
\sqrt{\bar{\alpha}}\mathcal {A}_{1,k+m-1}+\sqrt{\bar{\alpha}}\mathcal {A}_{2,k+m-1}\\
\sqrt{1-\bar{\alpha}}\mathcal {A}_{1,k+m-1}\\
\sqrt{\bar{\alpha}}\mathcal {C}_{1}+\sqrt{\bar{\alpha}}\mathcal {C}_{2,k+m-1}\\
\sqrt{1-\bar{\alpha}}\mathcal {C}_{1}
\end{array}
\right]\right)\varphi_{k+m-1,k}.
\end{align}}
Then, we verify the case of $l=k+m+1$. By definition,
\begin{align*}
\psi_{k+m+1,k} =\left[
\begin{array}{cccc}
\psi_{k+m+1,k+1}(\sqrt{\bar{\alpha}}\mathcal {A}_{1,k}+\sqrt{\bar{\alpha}}\mathcal {A}_{2,k})\\
\psi_{k+m+1,k+1}\sqrt{1-\bar{\alpha}}\mathcal {A}_{1,k}\\
\psi_{k+m+1,k+1}(\sqrt{\bar{\alpha}}\mathcal {C}_{1}+\sqrt{\bar{\alpha}}\mathcal {C}_{2,k})\\
\psi_{k+m+1,k+1}\sqrt{1-\bar{\alpha}}\mathcal {C}_{1}\\
\end{array}
\right].
\end{align*}
Because of the arbitrariness
of $k$,  we have
$\psi_{k+m+1,k+1}=\varphi_{k+m+1,k+1}.$
Therefore,
\begin{align*}
&\psi_{k+m+1,k}\\
=&\left(I_4\otimes \left[I_{4^{m-1}}\otimes\left[
\begin{array}{cccc}
\sqrt{\bar{\alpha}}\mathcal {A}_{1,k+m}+\sqrt{\bar{\alpha}}\mathcal {A}_{2,k+m}\\
\sqrt{1-\bar{\alpha}}\mathcal {A}_{1,k+m}\\
\sqrt{\bar{\alpha}}\mathcal {C}_{1}+\sqrt{\bar{\alpha}}\mathcal {C}_{2,k+m}\\
\sqrt{1-\bar{\alpha}}\mathcal {C}_{1}
\end{array}
\right]\right]\right)\\
&\left[\begin{array}{ccc}
\varphi_{k+m,k+1}(\sqrt{\bar{\alpha}}\mathcal {A}_{1,k}+\sqrt{\bar{\alpha}}\mathcal {A}_{2,k})\\ \varphi_{k+m,k+1}\sqrt{1-\bar{\alpha}}\mathcal {A}_{1,k}\\
\varphi_{k+m,k+1}(\sqrt{\bar{\alpha}}\mathcal
{C}_{1}+\sqrt{\bar{\alpha}}\mathcal {C}_{2,k})\\
\varphi_{k+m,k+1}\sqrt{1-\bar{\alpha}}\mathcal {C}_{1}
\end{array}\right]\\
=&\left(I_{4^{m}}\otimes\left[
\begin{array}{cccc}
\sqrt{\bar{\alpha}}\mathcal {A}_{1,k+m}+\sqrt{\bar{\alpha}}\mathcal {A}_{2,k+m}\\
\sqrt{1-\bar{\alpha}}\mathcal {A}_{1,k+m}\\
\sqrt{\bar{\alpha}}\mathcal {C}_{1}+\sqrt{\bar{\alpha}}\mathcal {C}_{2,k+m}\\
\sqrt{1-\bar{\alpha}}\mathcal {C}_{1}
\end{array}
\right]\right)\psi_{k+m,k}.
\end{align*}
So (10) is proved.  $\square$

\textbf {The proof of Lemma 3.2:}

We first prove (12). Because
$\alpha_k$ and $x_k$ are independent of each other, for $k=l-1$, we
have
\begin{align*}
&\mathcal {E}\|\tilde{x}_l\|^2
=\mathcal {E} \big[\big(\tilde{A}_k\tilde{x}_k+\tilde{C}_k\tilde{x}_kw_k\big)'\big(\tilde{A}_k\tilde{x}_k+\tilde{C}_k\tilde{x}_kw_k\big)\big]\\
=&\mathcal {E} \big\{\big[(\mathcal {A}_{1,k}+\alpha_k\mathcal {A}_{2,k})\tilde{x}_{k}+(\mathcal {C}_{1}+\alpha_k\mathcal {C}_{2,k})\tilde{x}_{k}w_k\big]'\\
&\big[(\mathcal {A}_{1,k}+\alpha_k\mathcal {A}_{2,k})\tilde{x}_{k}+(\mathcal {C}_{1}+\alpha_k\mathcal {C}_{2,k})\tilde{x}_{k}w_k\big]\big\}\\
=&\mathcal {E} \big[\tilde{x}_k'(\mathcal {A}_{1,k}'\mathcal {A}_{1,k}+\bar{\alpha}\mathcal {A}_{1,k}'\mathcal {A}_{2,k}+\bar{\alpha}\mathcal {A}_{2,k}'\mathcal {A}_{1,k}
+\bar{\alpha}\mathcal {A}_{2,k}'\mathcal {A}_{2,k})\tilde{x}_k\\
&+\tilde{x}_k'(\mathcal {C}_{1}'\mathcal {C}_{1}+\bar{\alpha}\mathcal {C}_{1}'\mathcal {C}_{2,k}+\bar{\alpha}\mathcal {C}_{2,k}'\mathcal {C}_{1}+\bar{\alpha}\mathcal {C}_{2,k}'\mathcal {C}_{2,k})\tilde{x}_k\big]\\
=&\mathcal {E}\|\psi_{k+1,k}\tilde{x}_k\|^2.
\end{align*}
Hence, equation (12) holds for $k=l-1$. Assume that $\mathcal {E}\|\tilde{x}_l\|^2=\mathcal
{E}\|\psi_{l,k}\tilde{x}_k\|^2$ holds for $k=l-m$, $1<m<l$. Now we
need to prove (12) in the case of $k=l-m-1$.
By Lemma~3.1, it can be
seen that
\begin{align*}
&\mathcal {E}\|\tilde{x}_{l}\|^2
=\mathcal {E} \|\psi_{l,l-m}\tilde{x}_{l-m}\|^2\\
=&\mathcal {E} \big[\tilde{x}_{l-m-1}'(\mathcal {A}_{1,l-m-1}'\psi_{l,l-m}'\psi_{l,l-m}\mathcal {A}_{1,l-m-1}\\
&+\bar{\alpha}\mathcal {A}_{1,l-m-1}'\psi_{l,l-m}'\psi_{l,l-m}\mathcal {A}_{2,l-m-1}\\
&+\bar{\alpha}\mathcal {A}_{2,l-m-1}'\psi_{l,l-m}'\psi_{l,l-m}\mathcal {A}_{1,l-m-1}\\
&+\bar{\alpha}\mathcal {A}_{2,l-m-1}'\psi_{l,l-m}'\psi_{l,l-m}\mathcal {A}_{2,l-m-1})\tilde{x}_{l-m-1}\\
&+\tilde{x}_{l-m-1}'(\mathcal {C}_{1}'\psi_{l,l-m}'\psi_{l,l-m}\mathcal {C}_{1}+\bar{\alpha}\mathcal {C}_{1}'\psi_{l,l-m}'\psi_{l,l-m}\mathcal {C}_{2,l-m-1}\\
&+\bar{\alpha}\mathcal {C}_{2,l-m-1}'\psi_{l,l-m}'\psi_{l,l-m}\mathcal {C}_{1}+\bar{\alpha}\mathcal {C}_{2,l-m-1}'\psi_{l,l-m}'\psi_{l,l-m}\\
&\cdot\mathcal {C}_{2,l-m-1})\tilde{x}_{l-m-1}\big]\\
=&\mathcal {E}\|\psi_{l,l-m-1}\tilde{x}_{l-m-1}\|^2.
\end{align*}
So (12) is shown.  Finally, we  prove
(13). By (10), we have
{\small \begin{align*}
\psi_{l,k}
=&\left[
\begin{array}{cccc}
\psi_{l,k+1}(\sqrt{\bar{\alpha}}\mathcal {A}_{1,k}+\sqrt{\bar{\alpha}}\mathcal {A}_{2,k})\\
\psi_{l,k+1}\sqrt{1-\bar{\alpha}}\mathcal {A}_{1,k}\\
\psi_{l,k+1}(\sqrt{\bar{\alpha}}\mathcal {C}_{1}+\sqrt{\bar{\alpha}}\mathcal {C}_{2,k})\\
\psi_{l,k+1}\sqrt{1-\bar{\alpha}}\mathcal {C}_{1}\\
\end{array}
\right]\\
=&\left(I_{4^{l-k-1}}\otimes\left[
\begin{array}{cccc}
\sqrt{\bar{\alpha}}\mathcal {A}_{1,l-1}+\sqrt{\bar{\alpha}}\mathcal {A}_{2,l-1}\\
\sqrt{1-\bar{\alpha}}\mathcal {A}_{1,l-1}\\
\sqrt{\bar{\alpha}}\mathcal {C}_{1}+\sqrt{\bar{\alpha}}\mathcal {C}_{2,l-1}\\
\sqrt{1-\bar{\alpha}}\mathcal {C}_{1}
\end{array}
\right]\right)\psi_{l-1,k}.
\end{align*}}
 Note that there must exist elementary matrices
$P_{l,k}\in\mathcal {R}^{(4^{l-k}2n)\times(4^{l-k}2n)}$ and
$P_{l-1,k}\in\mathcal {R}^{(4^{l-k-1}2n)\times(4^{l-k-1}2n)}$ with
$P_{l,k}'P_{l,k}=I_{4^{l-k}2n}$ and
$P_{l-1,k}'P_{l-1,k}=I_{4^{l-k-1}2n}$, such that
\begin{align*}
&I_{4^{l-k-1}}\otimes\left[\begin{array}{cccc}
\sqrt{\bar{\alpha}}\mathcal {A}_{1,l-1}+\sqrt{\bar{\alpha}}\mathcal {A}_{2,l-1}\\
\sqrt{1-\bar{\alpha}}\mathcal {A}_{1,l-1}\\
\sqrt{\bar{\alpha}}\mathcal {C}_{1}+\sqrt{\bar{\alpha}}\mathcal {C}_{2,l-1}\\
\sqrt{1-\bar{\alpha}}\mathcal {C}_{1}\end{array}\right]\\
=&P_{l,k}^{-1}\left[
\begin{array}{cc}
I_{4^{l-k-1}}\otimes (\sqrt{\bar{\alpha}}\mathcal {A}_{1,l-1}+\sqrt{\bar{\alpha}}\mathcal {A}_{2,l-1})\\
I_{4^{l-k-1}}\otimes \sqrt{1-\bar{\alpha}}\mathcal {A}_{1,l-1}\\
I_{4^{l-k-1}}\otimes (\sqrt{\bar{\alpha}}\mathcal {C}_{1}+\sqrt{\bar{\alpha}}\mathcal {C}_{2,l-1})\\
I_{4^{l-k-1}}\otimes \sqrt{1-\bar{\alpha}}\mathcal {C}_{1}\\
\end{array}
\right]P_{l-1,k}.
\end{align*}
From (12), we can get that
{\small\begin{align*}
&\psi_{l,k}=\left(I_{4^{l-k-1}}\otimes\left[
\begin{array}{cccc}
\sqrt{\bar{\alpha}}\mathcal {A}_{1,l-1}+\sqrt{\bar{\alpha}}\mathcal {A}_{2,l-1}\\
\sqrt{1-\bar{\alpha}}\mathcal {A}_{1,l-1}\\
\sqrt{\bar{\alpha}}\mathcal {C}_{1}+\sqrt{\bar{\alpha}}\mathcal {C}_{2,l-1}\\
\sqrt{1-\bar{\alpha}}\mathcal {C}_{1}
\end{array}
\right]\right)\psi_{l-1,k}\\
\Rightarrow&\psi_{l,k}=P_{l,k}^{-1}\left[
\begin{array}{cc}
I_{4^{l-k-1}}\otimes (\sqrt{\bar{\alpha}}\mathcal {A}_{1,l-1}+\sqrt{\bar{\alpha}}\mathcal {A}_{2,l-1})\\
I_{4^{l-k-1}}\otimes \sqrt{1-\bar{\alpha}}\mathcal {A}_{1,l-1}\\
I_{4^{l-k-1}}\otimes (\sqrt{\bar{\alpha}}\mathcal {C}_{1}+\sqrt{\bar{\alpha}}\mathcal {C}_{2,l-1})\\
I_{4^{l-k-1}}\otimes \sqrt{1-\bar{\alpha}}\mathcal {C}_{1}\\
\end{array}
\right]\\
&\ \ \ \ \ \ \ \ \ \cdot P_{l-1,k}\psi_{l-1,k}\\
\Rightarrow&P_{l,k}\psi_{l,k}=\left[
\begin{array}{cc}
I_{4^{l-k-1}}\otimes (\sqrt{\bar{\alpha}}\mathcal {A}_{1,l-1}+\sqrt{\bar{\alpha}}\mathcal {A}_{2,l-1})\\
I_{4^{l-k-1}}\otimes \sqrt{1-\bar{\alpha}}\mathcal {A}_{1,l-1}\\
I_{4^{l-k-1}}\otimes (\sqrt{\bar{\alpha}}\mathcal {C}_{1}+\sqrt{\bar{\alpha}}\mathcal {C}_{2,l-1})\\
I_{4^{l-k-1}}\otimes \sqrt{1-\bar{\alpha}}\mathcal {C}_{1}\\
\end{array}
\right]\\
&\ \ \ \ \ \ \ \ \ \ \ \ \ \cdot P_{l-1,k}\psi_{l-1,k}.
\end{align*}}
Set $\phi_{l,k}=P_{l,k}\psi_{l,k}$, then (13)
is proved.  $\square$

\textbf {The proof of Lemma 3.3:}

 Note that $R_{k+1}$ and $R_k$ are given positive definite
matrices. It is easy to see that
\begin{align*}
&\bar{x}_{k+1}
=\bar{R}_{k+1}^{\frac{1}{2}}\tilde{x}_{k+1}\\
=&(\bar{R}_{k+1}^{\frac{1}{2}}\mathcal {A}_{1,k}\mathcal
{R}_k^{-\frac{1}{2}}+\alpha_k\bar{R}_{k+1}^{\frac{1}{2}}\mathcal
{A}_{2,k}\bar{R}_k^{-\frac{1}{2}})\bar{x}_k\\
&+(\bar
{R}_{k+1}^{\frac{1}{2}}\mathcal {C}_{1}\bar
{R}_k^{-\frac{1}{2}}+\alpha_k\bar{R}_{k+1}^{\frac{1}{2}}\mathcal
{C}_{2,k}\bar{R}_k^{-\frac{1}{2}})\bar{x}_kw_k.
\end{align*}
So the  dynamic system of $\bar{x}_k$ is obtained. The STM
$\bar{\phi}_{l,k}$ can be given via  Lemma
3.2, so can  $\bar{\psi}_{l,k}$. The proof is completed. $\square$

\textbf {The proof of Lemma 3.4:}

By Lemma 2.1, $I_{4^{l-k}}\otimes \bar{R}_l$ can be broken down into $(I_{4^{l-k}}\otimes \bar{R}_l^{\frac{1}{2}})(I_{4^{l-k}}\otimes \bar
{R}_l^{\frac{1}{2}})$. So, the problem reduces into proving the following equation:
\begin{align}\label{2019finite-fdsds}
(I_{4^{l-k}}\otimes \bar{R}_{l}^{\frac{1}{2}})\phi_{l,k}
=\bar{\phi}_{l,k}\bar{R}_k^{\frac{1}{2}}.
\end{align}
For $l=k$, in view of $\phi_{k,k}=\bar{\phi}_{k,k}=I_{2n}$, we have
$ (I_1\otimes \bar{R}_k^{\frac{1}{2}})\phi_{k,k}=\bar
{R}_k^{\frac{1}{2}}=\bar{\phi}_{k,k}\bar{R}_k^{\frac{1}{2}}. $
Hence, (\ref{2019finite-fdsds}) holds for $l=k$. We suppose (\ref{2019finite-fdsds}) holds when $l=k+i-1$, i.e.,
$(I_{4^{i-1}}\otimes \bar
{R}_{k+i-1}^{\frac{1}{2}})\phi_{k+i-1,k}=\bar{\phi}_{k+i-1,k}\bar
{R}_k^{\frac{1}{2}}$, then only the equation $(I_{4^i}\otimes
\bar{R}_{k+i}^{\frac{1}{2}})\phi_{k+i,k}=\bar{\phi}_{k+i,k}\bar
{R}_k^{\frac{1}{2}}$ needs to be proved. By Lemma 2.1, it can be seen
that
{\footnotesize\begin{align*}
&(I_{4^i}\otimes \bar{R}_{k+i}^{\frac{1}{2}})\phi_{k+i,k}\\
=&\left[
\begin{array}{cccc}
(I_{4^{i-1}}\otimes (\sqrt{\bar{\alpha}}\bar{R}_{k+i}\mathcal {A}_{1,k+i-1}+\sqrt{\bar{\alpha}}\bar{R}_{k+i}\mathcal {A}_{2,k+i-1}))\phi_{k+i-1,k}\\
(I_{4^{i-1}}\otimes \sqrt{1-\bar{\alpha}}\bar{R}_{k+i}\mathcal {A}_{1,k+i-1})\phi_{k+i-1,k}\\
(I_{4^{i-1}}\otimes (\sqrt{\bar{\alpha}}\bar{R}_{k+i}\mathcal {C}_{1}+\sqrt{\bar{\alpha}}\bar{R}_{k+i}\mathcal {C}_{2,k+i-1}))\phi_{k+i-1,k}\\
(I_{4^{i-1}}\otimes \sqrt{1-\bar{\alpha}}\bar{R}_{k+i}\mathcal
{C}_{1})\phi_{k+i-1,k}
\end{array}
\right]\\
=&\left[
\begin{array}{cccc}
\Theta_1'&
\Theta_2'&
\Theta_3'&
\Theta_4'
\end{array}
\right]'\\
=&\bar{\phi}_{k+i,k}\bar{R}_k^{\frac{1}{2}},
\end{align*}}
where
{\footnotesize\begin{align*}
    \Theta_1=&(I_{4^{i-1}}\otimes (\sqrt{\bar{\alpha}}\bar{R}_{k+i}\mathcal {A}_{1,k+i-1}+\sqrt{\bar{\alpha}}\bar{R}_{k+i}\mathcal {A}_{2,k+i-1}))\\
    &\cdot(I_{4^{i-1}}\otimes \bar{R}_{k+i-1}^{-\frac{1}{2}})(I_{4^{i-1}}\otimes \bar{R}_{k+i-1}^{\frac{1}{2}})\phi_{k+i-1,k},\\
    \Theta_2=&(I_{4^{i-1}}\otimes (\sqrt{1-\bar{\alpha}}\bar{R}_{k+i}\mathcal {A}_{1,k+i-1}))(I_{4^{i-1}}\otimes \bar{R}_{k+i-1}^{-\frac{1}{2}})\\
    &\cdot(I_{4^{i-1}}\otimes \bar{R}_{k+i-1}^{\frac{1}{2}})\phi_{k+i-1,k},\\
    \Theta_3=&(I_{4^{i-1}}\otimes (\sqrt{\bar{\alpha}}\bar{R}_{k+i}\mathcal {C}_{1}+\sqrt{\bar{\alpha}}\bar{R}_{k+i}\mathcal {C}_{2,k+i-1}))(I_{4^{i-1}}\otimes \bar{R}_{k+i-1}^{-\frac{1}{2}})\\
    &\cdot(I_{4^{i-1}}\otimes \bar{R}_{k+i-1}^{\frac{1}{2}})\phi_{k+i-1,k},\\
    \Theta_4=&(I_{4^{i-1}}\otimes (\sqrt{1-\bar{\alpha}}\bar{R}_{k+i}\mathcal{C}_{1}))(I_{4^{i-1}}\otimes \bar{R}_{k+i-1}^{-\frac{1}{2}})(I_{4^{i-1}}\otimes \bar{R}_{k+i-1}^{\frac{1}{2}})\\
    &\cdot\phi_{k+i-1,k}.
\end{align*}}
This completes the proof. $\square$

\textbf {The proof of Lemma 3.5:}

We also use the induction principle to prove the lemma. Obviously, in the case of $k=l$, (15) is right.  Suppose that for $k=i<l$,
(15) holds, i.e.,
$
\psi'_{l,i}(I_{4^{l-i}}\otimes \bar{R}_l)\psi_{l,i}=\bar
{R}_i^{\frac{1}{2}}\bar{\psi}'_{l,i}\bar{\psi}_{l,i}\bar
{R}_i^{\frac{1}{2}}.
$
Then we only need to show
$
\psi'_{l,i-1}(I_{4^{l-i+1}}\otimes \bar
{R}_l)\psi_{l,i-1}=\bar{R}_{i-1}^{\frac{1}{2}}\bar{\psi}'_{l,i-1}\bar{\psi}_{l,i-1}\bar
{R}_{i-1}^{\frac{1}{2}}.
$
The right hand side of the above equation can be computed as
{\footnotesize \begin{align*}
&\bar{R}_{i-1}^{\frac{1}{2}}\bar{\psi}'_{l,i-1}\bar{\psi}_{l,i-1}\bar{R}_{i-1}^{\frac{1}{2}}\\
=&\bar{R}_{i-1}^{\frac{1}{2}}\left[
\begin{array}{cccc}
\bar{\psi}_{l,i}(\sqrt{\bar{\alpha}}\bar{R}_i^{\frac{1}{2}}\mathcal {A}_{1,i-1}\bar{R}_{i-1}^{-\frac{1}{2}}+\sqrt{\bar{\alpha}}\bar{R}_i^{\frac{1}{2}}\mathcal {A}_{2,i-1}\bar{R}_{i-1}^{-\frac{1}{2}})\\
\bar{\psi}_{l,i}(\sqrt{1-\bar{\alpha}}\bar{R}_i^{\frac{1}{2}}\mathcal {A}_{1,i-1}\bar{R}_{i-1}^{-\frac{1}{2}})\\
\bar{\psi}_{l,i}(\sqrt{\bar{\alpha}}\bar{R}_i^{\frac{1}{2}}\mathcal {C}_{1}\bar{R}_{i-1}^{-\frac{1}{2}}+\sqrt{\bar{\alpha}}\bar{R}_i^{\frac{1}{2}}\mathcal {C}_{2,i-1}\bar{R}_{i-1}^{-\frac{1}{2}})\\
\bar{\psi}_{l,i}(\sqrt{1-\bar{\alpha}}\bar{R}_i^{\frac{1}{2}}\mathcal {C}_{1}\bar{R}_{i-1}^{-\frac{1}{2}})\\
\end{array}
\right]'\\
&\cdot\left[
\begin{array}{cccc}
\bar{\psi}_{l,i}(\sqrt{\bar{\alpha}}\bar{R}_i^{\frac{1}{2}}\mathcal {A}_{1,i-1}\bar{R}_{i-1}^{-\frac{1}{2}}+\sqrt{\bar{\alpha}}\bar{R}_i^{\frac{1}{2}}\mathcal {A}_{2,i-1}\bar{R}_{i-1}^{-\frac{1}{2}})\\
\bar{\psi}_{l,i}(\sqrt{1-\bar{\alpha}}\bar{R}_i^{\frac{1}{2}}\mathcal {A}_{1,i-1}\bar{R}_{i-1}^{-\frac{1}{2}})\\
\bar{\psi}_{l,i}(\sqrt{\bar{\alpha}}\bar {R}_i^{\frac{1}{2}}\mathcal {C}_{1}\bar{R}_{i-1}^{-\frac{1}{2}}+\sqrt{\bar{\alpha}}\bar{R}_i^{\frac{1}{2}}\mathcal {C}_{2,i-1}\bar{R}_{i-1}^{-\frac{1}{2}})\\
\bar{\psi}_{l,i}(\sqrt{1-\bar{\alpha}}\bar{R}_i^{\frac{1}{2}}\mathcal {C}_{1}\bar{R}_{i-1}^{-\frac{1}{2}})\\
\end{array}
\right]\bar{R}_{i-1}^{\frac{1}{2}}\\
=&\left[
\begin{array}{cccc}
\bar{\psi}_{l,i}(\sqrt{\bar{\alpha}}\bar{R}_i^{\frac{1}{2}}\mathcal {A}_{1,i-1}+\sqrt{\bar{\alpha}}\bar {R}_i^{\frac{1}{2}}\mathcal {A}_{2,i-1})\\
\bar{\psi}_{l,i}(\sqrt{1-\bar{\alpha}}\bar{R}_i^{\frac{1}{2}}\mathcal {A}_{1,i-1})\\
\bar{\psi}_{l,i}(\sqrt{\bar{\alpha}}\bar{R}_i^{\frac{1}{2}}\mathcal {C}_{1}+\sqrt{\bar{\alpha}}\bar{R}_i^{\frac{1}{2}}\mathcal {C}_{2,i-1})\\
\bar{\psi}_{l,i}(\sqrt{1-\bar{\alpha}}\bar{R}_i^{\frac{1}{2}}\mathcal {C}_{1})\\
\end{array}
\right]'\\
&\cdot\left[
\begin{array}{cccc}
\bar{\psi}_{l,i}(\sqrt{\bar{\alpha}}\bar{R}_i^{\frac{1}{2}}\mathcal {A}_{1,i-1}+\sqrt{\bar{\alpha}}\bar{R}_i^{\frac{1}{2}}\mathcal {A}_{2,i-1})\\
\bar{\psi}_{l,i}(\sqrt{1-\bar{\alpha}}\bar{R}_i^{\frac{1}{2}}\mathcal {A}_{1,i-1})\\
\bar{\psi}_{l,i}(\sqrt{\bar{\alpha}}\bar{R}_i^{\frac{1}{2}}\mathcal {C}_{1}+\sqrt{\bar{\alpha}}\bar{R}_i^{\frac{1}{2}}\mathcal {C}_{2,i-1})\\
\bar{\psi}_{l,i}(\sqrt{1-\bar{\alpha}}\bar{R}_i^{\frac{1}{2}}\mathcal {C}_{1})\\
\end{array}
\right]\\
=&\psi'_{l,i-1}(I_{4^{l-i+1}}\otimes \bar{R}_i)\psi_{l,i-1}.
\end{align*}}
This lemma is proved. $\square$

\textbf {The proof of Theorem 3.1:}

We first prove {\textbf {(b)}$\Rightarrow${\textbf {(a)} in Theorem
3.1. By Lemma 3.3, LDMFS system
(2) is finite-time stabilizable with respect
to $(\epsilon_1, \epsilon_2, T, \{R_k\}_{k\in{\mathcal N}_{T}})$ iff the system (14) is
finite-time stable with respect to $(\epsilon_1, \epsilon_2, T, \{I_{2n})\}$. If
\begin{align}\label{2019finite-cxx1}
x_0'R_0x_0\leq \epsilon_1,
\end{align}
then, by Lemma 3.4, we have
\begin{align}\label{2019finite-cxx2}
\mathcal {E}(x_k'R_kx_k)=&\bar{x}'_0\bar{\phi}'_{k,0}\bar{\phi}_{k,0}\bar{x}_0
=\tilde{x}_0'\bar
{R}_0^{\frac{1}{2}}\bar{\phi}'_{k,0}\bar{\phi}_{k,0}\bar
{R}_0^{\frac{1}{2}}\tilde{x}_0\nonumber\\
=&\tilde{x}_0'\phi'_{k,0}(I_{4^{k}}\otimes
\bar{R}_k)\phi_{k,0}\tilde{x}_0.
\end{align}
When $x_0=0$, it directly leads to $\mathcal
{E}(x_k'Rx_k)\equiv0<\epsilon_2$. When $x_0\neq0$, by
(\ref{2019finite-cxx1}) and (\ref{2019finite-cxx2}), we obtain
$
\mathcal {E}(x_k'R_kx_k)\leq \tilde{x}_0'\frac{\epsilon_2}{\epsilon_1}\bar
{R}_0\tilde{x}_0\leq \epsilon_2, k\in {\mathcal N}_T.
$
All in all, for any $x_0\in\mathcal{R}^n$ satisfying
(\ref{2019finite-cxx1}), it can be always concluded that $\mathcal
{E}(x_k'R_kx_k)\leq \epsilon_2$. Thus {\textbf {(a)} holds.

{\textbf {(a)}$\Rightarrow${\textbf {(b)}: In the case of that
system (2) is finite-time stabilizable with
respect to $(\epsilon_1, \epsilon_2, T, \{R_k\}_{k\in{\mathcal N}_{T}})$, (16)
must hold, otherwise,  there exist $k_0\in\mathcal {N}_T$ and
$\tilde{x}_0$ with $\tilde{x}_0'\bar {R}_0\tilde{x}_0=\epsilon_1$,
such  that
$$
\mathcal
{E}(\tilde{x}_{k_0}'\bar{R}_{k_0}\tilde{x}_{k_0})=\tilde{x}_0'\phi_{k_0,0}'(I_{4^{k_0}}\otimes
\bar
{R}_{k_0})\phi_{k_0,0}\tilde{x}_0\geq\tilde{x}_0'\frac{\epsilon_2}{\epsilon_1}\tilde{x}_0\geq
\epsilon_2.
$$
This contradicts the definition of finite-time stability. Therefore, {\textbf {(b)} is
derived.

By considering Lemma 3.4 and {\textbf {(b)}, we can get that
$\phi'_{k,0}(I_{4^{k}}\otimes \bar{R}_k)\phi_{k,0}=\bar
{R}_0^{\frac{1}{2}}\bar{\phi}'_{k,0}\bar{\phi}_{k,0}\bar
{R}_0^{\frac{1}{2}}$. Therefore {\textbf {(b)}$\Leftrightarrow${\textbf {(c)}. Analogously,
according to Lemmas 3.4 and 3.5,
it follows that {\textbf {(a)}$\Leftrightarrow${\textbf {(d)}$\Leftrightarrow${\textbf {(e)}.

{\textbf {(b)}$\Leftrightarrow${\textbf {(f)}: By Lemma 2.1 and
$\bar{R}_k>0$, we have    $I_{4^k}\otimes\bar
{R}_k=(I_{4^k}\otimes \bar{R}_k^{1/2})(I_{4^k}\otimes \bar{R}_k^{1/2})$, where $\bar{R}_k^{1/2}>0$. Considering {\textbf {(b)}, it can be obtained that
\begin{align*}
&\phi_{k,0}'(I_{4^k}\otimes \bar{R}_k)\phi_{k,0}\leq\frac{\epsilon_2}{\epsilon_1}\bar{R}_0\\
\Leftrightarrow&\phi_{k,0}'(I_{4^k}\otimes \bar{R}_k^{1/2})(I_{4^k}\otimes \bar{R}_k^{1/2})\phi_{k,0}\leq\frac{\epsilon_2}{\epsilon_1}\bar{R}_0\\
\Leftrightarrow&\bar{R}_0^{-1/2}\phi_{k,0}'(I_{4^k}\otimes
\bar{R}_k^{1/2})(I_{4^k}\otimes \bar
{R}_k^{1/2})\phi_{k,0}\bar{R}_0^{-1/2}\leq\frac{\epsilon_2}{\epsilon_1}I_{2n}.
\end{align*}
Through Schur's complement lemma, the above relationship yields that
{\small
\begin{align*}
&\left[
\begin{array}{ccccc}
-\frac{\epsilon_2}{\epsilon_1}I_{2n}&\bar{R}_0^{-1/2}\phi_{k,0}'(I_{4^k}\otimes \bar{R}_k^{1/2})\\
(I_{4^k}\otimes \bar{R}_k^{1/2})\phi_{k,0}\bar
{R}_0^{-1/2}&-I_{4^k\times2n}
\end{array}
\right]\leq0\\
\Leftrightarrow&[(I_{4^k}\otimes \bar
{R}_k^{1/2})\phi_{k,0}\bar{R}_0^{-1/2}][\bar
{R}_0^{-1/2}\phi'_{k,0}(I_{4^k}\otimes \bar
{R}_k^{1/2})]\\
&-\frac{\epsilon_2}{\epsilon_1}I_{4^k\times2n}\leq0.
\end{align*}}
Set $P_k=\phi_{k,0}\bar{R}_0^{-1}\phi'_{k,0}$, then {\textbf {(b)} is
equivalent  to
\begin{align}\label{2019finite-gfbvcv}
(I_{4^k}\otimes \bar{R}_k^{1/2})P_k(I_{4^k}\otimes \bar
{R}_k^{1/2})-\frac{\epsilon_2}{\epsilon_1}I_{4^k\times2n}\leq0.
\end{align}
Pre-multiplying  and post-multiplying  (\ref{2019finite-gfbvcv}) by
$(I_{4^k}\otimes\bar{R}_k^{-1/2})=(I_{4^k}\otimes\bar
{R}_k^{-1/2})'$,  it can be seen that (\ref{2019finite-gfbvcv}) is
equivalent  to
$
P_k-\frac{\epsilon_2}{\epsilon_1}(I_{4^k}\otimes \bar{R}_k^{-1})\leq0.
$
In addition, by the definition of $P_k$,  it is obvious that
$P_0=\bar{R}_0^{-1}$. Hence, in order to show
{\textbf {(b)}$\Leftrightarrow${\textbf {(f)}, we only need to show that $P_k$   satisfies
\begin{align*}
P_{k+1}=&\left[\begin{array}{cccccc}
I_{4^{k}}\otimes (\sqrt{\bar{\alpha}}\mathcal {A}_{1,k}+\sqrt{\bar{\alpha}}\mathcal {A}_{2,k})\\
I_{4^{k}}\otimes \sqrt{1-\bar{\alpha}}\mathcal {A}_{1,k}\\
I_{4^{k}}\otimes (\sqrt{\bar{\alpha}}\mathcal {C}_{1}+\sqrt{\bar{\alpha}}\mathcal {C}_{2,k})\\
I_{4^{k}}\otimes \sqrt{1-\bar{\alpha}}\mathcal {C}_{1}
\end{array}\right]P_k\\
&\left[\begin{array}{cccccc}
I_{4^{k}}\otimes (\sqrt{\bar{\alpha}}\mathcal {A}_{1,k}+\sqrt{\bar{\alpha}}\mathcal {A}_{2,k})\\
I_{4^{k}}\otimes \sqrt{1-\bar{\alpha}}\mathcal {A}_{1,k}\\
I_{4^{k}}\otimes (\sqrt{\bar{\alpha}}\mathcal {C}_{1}+\sqrt{\bar{\alpha}}\mathcal {C}_{2,k})\\
I_{4^{k}}\otimes \sqrt{1-\bar{\alpha}}\mathcal {C}_{1}
\end{array}\right]',
\end{align*}
which can be derived by applying
 the expression of
$\phi_{k,0}$ in Lemma 3.2 and the definition of
$P_k=\phi_{k,0}\bar{R}_0^{-1}\phi'_{k,0}$.
The proof is ended.  $\square$

\end{document}